\documentclass{article}
\usepackage[utf8]{inputenc}
\usepackage{amsmath}
\usepackage{amssymb}
\usepackage{amsthm}
\usepackage{xcolor}
\usepackage{dsfont}
\usepackage{authblk}
\usepackage{bm}
\usepackage{bbm}
\usepackage{graphicx}
\usepackage{bbding}
\usepackage{pifont}
\usepackage{wrapfig}
\usepackage{hyperref}
\usepackage{algorithm2e}
\usepackage{chemfig}
\RestyleAlgo{ruled}
\usepackage{biblatex}  
\usepackage{tikz}
\usetikzlibrary{calc}
\usetikzlibrary{shapes.geometric}
\usepackage{subcaption}
\usepackage{appendix}
\usepackage{comment}
\usepackage{tablefootnote}
\usepackage{empheq}
\usepackage{float}
\usepackage[symbol]{footmisc}

\title{A quadratic Grassmann manifold optimization problem arising from quantum embedding methods}
\author[1]{Thomas Ayral}
\author[2,3]{Eric Canc\`es}
\author[4]{Fabian M. Faulstich}
\author[5,6]{Lin Lin} 
\author[2,3]{$\mbox{Alicia Negre}$}
\affil[1]{CPHT, CNRS, Ecole Polytechnique, IP Paris, F-91128 Palaiseau, France}
\affil[2]{$\mbox{CERMICS, ENPC, IP Paris, CNRS, F-77455 Marne-la-Vall\'ee, France}$}
\affil[3]{MATHERIALS team-project, Inria Paris, France}
\affil[4]{Department of Mathematical Sciences, Rensselaer Polytechnic Institute, $\mbox{Troy, NY 12180, USA}$}
\affil[5]{Department of Mathematics, University of California, Berkeley, CA 94720, USA}
\affil[6]{Applied Mathematics and Computational
Research Division, Lawrence Berkeley National Laboratory, Berkeley, CA 94720, USA}

\def\N{\mathbb N}
\def\Z{\mathbb Z}
\def\R{\mathbb R}

\newtheorem{theorem}{Theorem}
\newtheorem{remark}[theorem]{Remark}
\newtheorem{lemma}[theorem]{Lemma}
\newtheorem{proposition}[theorem]{Proposition}

\newcommand{\tr}{{\rm Tr}}

\newcommand{\1}{{\mathds 1}}

\begin{document}

\maketitle

\begin{abstract}
This article presents a mathematical analysis and numerical strategies for solving the optimization problem of minimizing the quadratic function $J(P) = \text{Tr}(BP)- \frac{1}{2} \text{Tr}(A P A P)$,  where $A,B \in \R^{M \times M}_{\rm sym}$, with $A \succeq 0$, over the Grassmann manifold ${\rm Gr}(m,\R^M)$. While this problem is non-convex and typically admits non-global local minima - posing challenges for Riemannian optimization and self-consistent field (SCF) algorithms - we identify cases where the global minimizer can be obtained by solving an auxiliary convex problem. When this approach is not directly applicable, the solution to the auxiliary problem still serves as an effective initialization for Riemannian optimization methods and SCF algorithms, significantly improving their performance. This work is motivated by applications in quantum embedding methods, particularly in the construction of bath orbitals, where such optimization problems naturally arise.
\end{abstract}

\setcounter{tocdepth}{2}               
\tableofcontents

\section{Introduction}

This article is concerned with the mathematical analysis and numerical treatment of the quadratic Grassmann optimization problem 
\begin{equation}
\label{eq:OptProblem}
\min_{P \in \mathcal M} J(P), 
\end{equation}
where $\mathcal M$ is the real Grassmann manifold
\begin{equation}\label{eq:Grassmann_mM}
\mathcal M:= {\rm Gr}(m,\R^M):=\left\{ P \in \R^{M \times M}_{\rm sym} \; | \; P^2=P, \; \tr(P)=m \right\},
\end{equation}
(the set of rank-$m$ orthogonal projectors in $\R^M$) and $J:\R^{M\times M}_{\rm sym} \to \R$ the quadratic function defined by
\begin{equation}
\label{eq:CostFunctionJ}
    J(P) = \text{Tr}(BP)- \frac{1}{2} \text{Tr}(A P A P) = \text{Tr}(BP) - \frac{1}{2} \lVert A^{\frac{1}{2}} P A^{\frac{1}{2}} \rVert^2 \quad \mbox{with} \quad A,B \in \R^{M \times M}_{\rm sym}, \quad A \succeq 0,
\end{equation}
where $\|\cdot\|$ denotes the Frobenius norm.
This quadratic Grassmann optimization problem arises naturally in embedding methods used in quantum chemistry and physics. In~\cite{negre2025DMET}, we revisit and extend the specific quantum embedding method called {\it Density-Matrix Embedding Theory} (DMET)~\cite{DMET2012}. In this extension, the generalized bath construction leads to an optimization problem of the form~\eqref{eq:OptProblem}--\eqref{eq:CostFunctionJ}.

\medskip

To the best of our knowledge, this problem has not been studied in the mathematical literature. Beyond its application in DMET, the problem~\eqref{eq:OptProblem}--\eqref{eq:CostFunctionJ} is of independent interest for several reasons:
\begin{itemize}
\item Grassmann optimization is a textbook example to illustrate Riemannian optimization on quotient manifolds (see e.g.~\cite[Chapter 3]{absil2008optimization}, ~\cite[Chapter 9]{boumal2023introduction}). Minimizing a linear functional $P \mapsto \text{Tr}(BP)$ over the Grassmann manifold ${\rm Gr}(m;\R^M)$ reduces to computing the lowest $m$ eigenmodes of the symmetric matrix $B \in \R^{M \times M}_{\rm sym}$. The quadratic function $J$ is one of the simplest functionals beyond the trivial case of a linear functional;
    \item problem~\eqref{eq:OptProblem}--\eqref{eq:CostFunctionJ} is not a generic quadratic Grassmann optimization problem (it has been recently proved that the latter class of problems is NP-hard~\cite{Lai2025}). As discussed below, it has specific mathematical properties that can be exploited in numerical methods;
    \item problem ~\eqref{eq:OptProblem}--\eqref{eq:CostFunctionJ} is reminiscent to the Brockett problem~\cite{boumal2023} in which the functional $J$ defined in~\eqref{eq:CostFunctionJ} is minimized over the Lie group $O(M)$.
\end{itemize}

This article is organized as follows. The mathematical properties of problem \eqref{eq:OptProblem}--\eqref{eq:CostFunctionJ} are presented Section~\ref{sec:analysis}. After recalling the first and second-order optimality conditions in Section~\ref{sec:opt_cond}, we show in Section~\ref{sec:Aufbau} that any global minimizer of~\eqref{eq:OptProblem}--\eqref{eq:CostFunctionJ} satisfies the {\em Aufbau principle}, a property shared by the (unrestricted) Hartree--Fock problem~\cite{BLLS94} but not by all quadratic Grassmann optimization problems. In Section~\ref{sec:convexification}, we introduce a convexification of problem~\eqref{eq:OptProblem}--\eqref{eq:CostFunctionJ}, i.e., 

\begin{equation}
\label{eq:convexification}
\inf_{D \in {\rm CH}(\mathcal M)} \widetilde J(D),
\end{equation}

\noindent 
where ${\rm CH}(\mathcal M)$ is the convex hull of $\mathcal M$, and $\widetilde J: \R^{M \times M}_{\rm sym} \to \R$ is defined by 

\begin{equation}
\widetilde J(D):=  \tr(CD)+ \frac 14 \| [A,D] \|^2  \quad \mbox{with} \quad C := B- \frac 12A^2.
\end{equation}
Note that $J\big|_\mathcal{M} = \widetilde J\big|_\mathcal{M}$, while $J\neq \widetilde J$ on $\mathcal{\R^{M \times M}_{\rm sym}}$ unless $A=0$. Solving this convex problem numerically allows one to compute the gradient $\widetilde G_\star \in \R^{M \times M}_{\rm sym}$ of $\widetilde J$ at the found minimizer. The matrix $\widetilde G_\star$ is actually independent of the specific minimizer attained by the algorithm used. In addition, if it satisfies a spectral gap condition (vide infra), then \eqref{eq:convexification} admits a unique solution, this solution lies on $\mathcal M$, and, remarkably, it is the {\em global} minimizer of the non-convex problem \eqref{eq:OptProblem}--\eqref{eq:CostFunctionJ}.
This property extends the well-known result that, if there is a gap between the $m^{\rm th}$ and $(m+1)^{\rm st}$ eigenvalues of $B$, then the minimizer of $D \mapsto \tr(BD)$ over ${\rm CH}(\mathcal M)$ is unique, belongs to $\mathcal M$, and is the global minimizer over $\mathcal M$. In Section~\ref{sec:special}, we study three special cases: (i) A and B commute, in which case an explicit solution of~\eqref{eq:OptProblem}--\eqref{eq:CostFunctionJ} can be derived from the eigenmodes of the matrix C; (ii) the commutator $[A,B]$ is small, i.e., a perturbation of case (i); and (iii) the matrices A and B arise from a DMET bath construction problem. In the latter case, $A$ and $B$ arise as blocks in a suitable block decomposition of some trial one-body reduced density matrix (1-RDM)
\begin{align*}
     \gamma = \begin{pmatrix}
        \gamma_{\rm frag } & \gamma_{{\rm frag}, \rm ext}\\
        \gamma_{{\rm ext}, \rm frag} & \gamma_{\rm ext} \end{pmatrix}, \quad \gamma_{\rm frag } \in \R^{\ell \times \ell}_{\rm sym}, \quad   \gamma_{\rm ext} \in \R^{M \times M}_{\rm sym}, \quad \gamma_{{\rm frag}, \rm ext} = \gamma_{{\rm ext}, \rm frag}^T \in \mathbb{R}^{\ell \times M},
\end{align*}
where the subscripts ``frag'' and  ``ext'' denote the fragment and exterior of the fragment, respectively. The 1-RDM $\gamma$ satisfies $0 \preceq \gamma  \preceq I$, and defines the matrices $A$ and $B$ through the formulae
\begin{align}
\label{rdm1_AB}
A=\gamma_{\rm ext} \in \R^{M \times M}_{\rm sym}
\qquad {\rm and} \qquad 
B=\frac 12 \left(\gamma_{\rm ext}^2-\gamma_{\rm ext,frag} \gamma_{\rm frag,ext}\right)\in \R^{M \times M}_{\rm sym}.
\end{align}
We derive necessary and sufficient conditions for {\em full disentanglement} at the 1-RDM level of theory, a concept precisely defined in~Section~\ref{sec:DMET}. Since this notion is relevant to quantum embedding methods, we compare our approach with those reported in the physical chemistry literature.
In Section~\ref{sec:numerical_methods}, we detail the numerical methods used to solve the original problem \eqref{eq:OptProblem}--\eqref{eq:CostFunctionJ} or its convex relaxation~\eqref{eq:convexification}. Finally, in Section~\ref{sec:numerics}, we present numerical results. We begin by considering illustrative and simple examples using $2 \times 2$ or $3 \times 3$  matrices $A$ and $B$ (Section~\ref{sec:random}). We then present examples arising from the DMET bath-construction problem for the benzene molecule (Section~\ref{sec:molecules}) and the Hubbard model (Section~\ref{sec:Hubbard}).

\section{Mathematical analysis}
\label{sec:analysis}

In this analysis, the Grassmann manifold $\mathcal M$ is viewed as a Riemannian submanifold of $\mathbb R^{M \times M}_{\rm sym}$ endowed with the metric induced by the Frobenius inner product $\langle Y,X \rangle := \tr(X^TY)$.

\subsection{First and second order optimization conditions}
\label{sec:opt_cond}

For completeness and the reader's convenience, this section collects several known results.
The tangent space to $\mathcal M$ at $P\in\mathcal M$ is
\begin{align}
  \mathcal T_P \mathcal M
&= \{ X\in\mathbb R^{M\times M}_{\mathrm{sym}} \,:\, XP + PX = X \} \nonumber \\
&= \{ X = [\Omega, P] : \Omega^T = -\Omega \in \R^{M \times M}_{\rm antisym} \} \nonumber \\
&=  \{ X\in\mathbb R^{M\times M}_{\mathrm{sym}} \,:\, PXP =(1-P)X(1-P) = 0 \} . \label{eq:tangent_space_def}
\end{align}    
The orthogonal projection of a matrix $X$ in $\mathbb{R}_{\rm sym}^{M\times M}$ onto $\mathcal{T}_{P}\mathcal{M}$ has an explicit expression:
    \begin{equation}
    \label{projector_tangent}
        \Pi_{\mathcal{T}_{P}\mathcal M}(X) = [[X,P], P], \qquad \mbox{ for all } P \in \mathcal M, \; X \in \mathbb{R}_{\rm sym}^{M\times M}.
    \end{equation}  
Let $J : \R^{M \times M}_{\rm sym} \to \R$ be the cost function introduced in \eqref{eq:CostFunctionJ} and 
\begin{equation}
\label{eq:gradient}
    G(P):=\nabla J(P)= B-AP A
\end{equation}
the gradient of $J$ at point $P \in \R^{M \times M}_{\rm sym}$. The Riemannian gradient ${\rm grad}_{\mathcal M}J(P)$ of the restriction of $J$ to $\mathcal M$ at some point $P \in \mathcal M$ is given by
\begin{align}
    \label{eq:Riemann_gradient}
 {\rm grad}_{\mathcal M}J(P)= \Pi_{\mathcal{T}_{P}\mathcal M}(\nabla J(P)) = [[G(P),P],P].   
\end{align}
Equation~\eqref{eq:Riemann_gradient} provides a natural extension 
$$
\mathcal G : \R^{M \times M}_{\rm sym} \ni P \mapsto \mathcal G(P):=[[G(P),P],P] \in  \R^{M \times M}_{\rm sym}
$$
of the vector field ${\rm grad}_{\mathcal M}J(P)$, originally defined on $\mathcal M$, to the whole space $\R^{M \times M}_{\rm sym}$.
This allows us to easily compute the Riemannian Hessian of $J$ at $P$ as 
\begin{align}\label{eq:Hessian}
\forall X \in \mathcal{T}_{P}\mathcal{M}, \quad \mbox{Hess}_{\mathcal{M}} J(P)[X] = \Pi_{\mathcal{T}_P\mathcal{M}} \big(D\mathcal G(P)[X]\big),    
\end{align}
where $D\mathcal G(P) \in \mathcal L(\R^{M \times M}_{\rm sym})$ is the derivative of $\mathcal G$ at $P \in \mathcal M$. 
Since $J$ is quadratic, its Hessian is independent of $P$, and explicitly given by
$$
\forall P \in \mathcal M, \quad \forall X \in \R^{M \times M}_{\rm sym}, \quad D^2J(P)[X]=-AXA.
$$
Combining this with~\eqref{projector_tangent} this yields
\begin{align} \label{eq:DGtilde}
D\mathcal G(P)[X] &= [[D^2J(P)[X],P],P] + [[G(P), X],P] + [[G(P),P],X].
\end{align}
Note that if 
$$
P = 
U \left( \begin{array}{cc} I_m & 0 \\ 0 & 0 \end{array} \right)U^T \in \mathcal M 
\quad \mbox{and} \quad 
X = 
U \begin{pmatrix} X_{\rm oo} & X_{\rm ov} \\ X_{\rm vo} &  X_{\rm vv} \end{pmatrix} U^T \in \R^{M \times M},
$$
then 
\begin{align} 
\label{eq:proj_TPM}
[X,P] 
=
U \begin{pmatrix} 0 & -X_{\rm ov} \\ X_{\rm vo} & 0\end{pmatrix} U^T 
\qquad {\rm and} \qquad  
\Pi_{T_P\mathcal M} X 
= [[X,P],P] 
= U \begin{pmatrix} 0 & X_{\rm ov} \\ X_{\rm vo} & 0 \end{pmatrix} U^T.
\end{align}
Moreover, $X \in T_P\mathcal M$ iff $X_{\rm oo}= 0$ and $X_{\rm vv}=0$.
The subscripts $\rm o$ and $\rm v$ stand for ``occupied'' and ``virtual''. In mean-field quantum chemistry models such as Hartree--Fock and Kohn--Sham,  the matrix $P$ represents the one-body density matrix  of the quantum state, and decomposing $\mathbb R^M$ as
   \begin{align*}
        \mathbb{R}^{M} = \mathrm{Ran}(P) \oplus \ker(P)
    \end{align*}
    the occupied orbitals are the elements in $\mathrm{Ran}(P)$ and the virtual orbitals the elements in ${\rm Ker}(P)$.

We deduce from~\eqref{eq:Hessian}--\eqref{eq:proj_TPM} that
\begin{align*} 
\forall X \in \mathcal{T}_{P}\mathcal{M}, \quad \mbox{\rm Hess}_{\mathcal{M}} J(P)[X] &= [[D^2J(P)[X],P],P] + [[G(P), X],P] + [[[[G(P),P],X],P],P] \nonumber \\
&= -[[AXA,P],P] + [[G(P), X],P].
\end{align*}
Following the notation from~\cite{Cances_2021}, we finally obtain
\begin{align} \label{eq:hessian2}
 \mbox{\rm Hess}_{\mathcal{M}} J(P) &=  \Omega(P) + K(P),
 \end{align}
 with 
 \begin{equation}
\label{eq:omega_2ndorder}
\begin{array}{rcl}
    \Omega(P) :  \mathcal{T}_{P} \mathcal M  & \to &  \mathcal{T}_{P} \mathcal M \\
 X & \mapsto &  [[G(P), X],P] \end{array} \qquad 
 \begin{array}{rcl}
    K(P) :  \mathcal{T}_{P} \mathcal M  & \to &  \mathcal{T}_{P} \mathcal M \\
 X & \mapsto &  -[[AXA,P],P] \end{array}.
\end{equation}
The term $K(P)\equiv \Pi_{\mathcal{T}_{P}\mathcal M}  (D^2J(P)) \Pi_{\mathcal{T}_{P}\mathcal M}$ is simply the projection onto the tangent space of the Hessian of the function $J$, while the term $\Omega(P)$ originates from the curvature of the Grassmann manifold $\mathcal M$.

\begin{proposition}[Optimality conditions] \label{Prop:commutation1storder} Let 
$P_\star \in \mathcal M$ and $U \in O(M)$ such that
\begin{align}
    P_\star = U\begin{pmatrix} I_m & 0 \\ 0 & 0\end{pmatrix}U^T, \qquad  
    A = U \begin{pmatrix} A_{\rm oo} & A_{\rm ov} \\ A_{\rm vo} & A_{\rm vv}\end{pmatrix}U^T, \qquad B = U\begin{pmatrix} B_{\rm oo} & B_{\rm ov} \\ B_{\rm vo} & B_{\rm vv}\end{pmatrix}U^T. \label{eq:Pstar_A_B}
\end{align}
\begin{itemize}
\item First-order condition: $P_\star$ is a critical point of $J$ on $\mathcal M$ if and only if 
\begin{align}
\label{eq:commutation_gradP}
    [G(P_\star),P_\star] = [B-AP_\star A,P_\star] =0 \quad \mbox{or equivalently} \quad B_{\rm ov}=A_{\rm oo}A_{\rm ov}.
\end{align}
\item Second-order condition: if $P_\star$ satisfies~\eqref{eq:commutation_gradP}, a necessary condition for $P_\star$ being a local minimizer of $J$ on $\mathcal M$ is 
\begin{align} \label{eq:2nd_order_cond}
  \forall X =   U\begin{pmatrix} 0 & X_{\rm ov} \\ X_{\rm vo} & 0\end{pmatrix}U^T \in \mathcal{T}_{P_\star}\mathcal{M}, \quad \langle X, \mbox{\rm Hess}_{\mathcal{M}} J(P_*)[X] \rangle \ge 0,
\end{align}
or equivalently 
$$
 \tr\left( G_{\rm vv}(X_{\rm ov}^TX_{\rm ov})\right)- \tr \left(G_{\rm oo}(X_{\rm ov}X_{\rm ov}^T)\right) - \tr \left( X_{\rm ov}  A_{\rm vo }X_{\rm ov} A_{\rm vo}+ A_{\rm oo}X_{\rm ov}A_{\rm vv} X_{\rm ov}^T \right) \ge 0,
$$
with $G_{\rm oo}=B_{\rm oo}-A_{\rm oo}^2$ and $G_{\rm vv}=B_{\rm vv}-A_{\rm vo}A_{\rm ov}$. 
\end{itemize}
\end{proposition}

\begin{proof} The first assertion is a direct consequence of~\eqref{eq:gradient},~\eqref{eq:Riemann_gradient} and \eqref{eq:proj_TPM}, and of the fact that
\begin{align*}
    G(P_\star) = \begin{pmatrix} G_{\rm oo}  & G_{\rm ov} \\ G_{\rm vo}  & G_{\rm vv} \end{pmatrix} = \begin{pmatrix} B_{\rm oo} - A_{\rm oo}^2 & B_{\rm ov}-A_{\rm oo}A_{\rm ov} \\ B_{\rm vo} - A_{\rm vo}A_{\rm oo} & B_{\rm vv}-A_{\rm vo}A_{\rm ov} \end{pmatrix}.
\end{align*}
From~\eqref{eq:Pstar_A_B}, we obtain
$$
AXA =  U \begin{pmatrix}
        * &  A_{\rm ov}X_{\rm vo}A_{\rm ov} + A_{\rm oo} X_{\rm ov} A_{\rm vv} \\  A_{\rm vo}X_{\rm ov}A_{\rm vo} + A_{\rm vv} X_{\rm vo} A_{\rm oo} & *  \end{pmatrix} U^T.
$$
and thus 
\begin{align*}
    K(P_\star)[X] = U \begin{pmatrix}
        0 &  -(A_{\rm ov}X_{\rm vo}A_{\rm ov} + A_{\rm oo} X_{\rm ov} A_{\rm vv}) \\  -(A_{\rm vo}X_{\rm ov}A_{\rm vo} + A_{\rm vv} X_{\rm vo} A_{\rm oo}) & 0  \end{pmatrix} U^T.
\end{align*}
On the other hand, using the fact that $[G(P_\star),P_\star]=0$ is equivalent to $G_{\rm ov}=G_{\rm vo}^T=0$, we get
\begin{align*}
    \Omega(P_\star)[X] = U \begin{pmatrix}
     0   & X_{\rm ov} G_{\rm vv} - G_{\rm oo} X_{\rm ov} \\  G_{\rm vv} X_{\rm vo}  - X_{\rm vo}G_{\rm oo}  & 0  
    \end{pmatrix} U^T.
\end{align*}
This yields
\begin{align*}
  \langle X, \mbox{\rm Hess}_{\mathcal{M}} J(P_\star)[X] \rangle & = 2 \tr\left(X_{\rm ov}^T  (X_{\rm ov} G_{\rm vv} - G_{\rm oo} X_{\rm ov} )\right) - 2 \tr \left( X_{\rm ov}^T  (A_{\rm ov}X_{\rm vo}A_{\rm ov} + A_{\rm oo} X_{\rm ov} A_{\rm vv}) \right).
\end{align*}
Rearranging the terms in the traces, we obtain the expected result.
\end{proof}

\subsection{Aufbau principle}
\label{sec:Aufbau}

The following property turns out to be very useful to design fixed-point iterative algorithms to solve the Euler-Lagrange equations associated to the constrained optimization problem~\eqref{eq:OptProblem}--\eqref{eq:CostFunctionJ} .

\begin{theorem} \label{thm:Aufbau}
    Let $A \in \R^{M \times M}_{\rm sym}$ be a positive semidefinite matrix, $B \in \R^{M \times M}_{\rm sym}$, and $P_\star$ a global minimizer of \eqref{eq:OptProblem} with $J$ given by \eqref{eq:CostFunctionJ}. Then, there exists an orthonormal basis $(x_k)_{1 \le k \le M}$ of $\R^M$ diagonalizing both $P_\star$ and $G(P_\star)$ such that
   \begin{align}
   & P_\star = \sum_{k=1}^m x_kx_k^T, \quad G(P_\star) = \sum_{k=1}^M \lambda_k x_kx_k^T,   \quad x_k^Tx_l=\delta_{kl}, \quad \lambda_1 \le \lambda_2 \le \cdots \le \lambda_M.  \label{eq:Aufbau}
\end{align}
In addition, if $A$ is invertible, then 
\begin{equation} \label{eq:no_unfilled_shell}
\lambda_m < \lambda_{m+1} \quad \mbox{and} \quad    P_\star=\mathds 1_{(-\infty,\mu]}(G(P_\star)) \quad \mbox{for any } \mu \in [\lambda_m ,\lambda_{m+1}).
\end{equation}
\end{theorem}

In words, the above results mean that $P_\star$ is the orthogonal projector on an $m$-dimensional vector subspace of $\R^M$ spanned by eigenvectors of $G(P_\star)$ associated with the {\em lowest} $m$ eigenvalues. A similar property is satisfied by the (unrestricted) Hartree--Fock problem in quantum chemistry (see e.g. \cite{szabo2012modern}, or~\cite{Bach:2022ewx} for a mathematical presentation). In that case, the Grassmann manifold is the set of admissible Hartree--Fock density matrices and the cost function is the Hartree--Fock energy functional. That function is also the sum of a linear term and a quadratic term, the latter being itself the sum of two terms: the Hartree term and the exchange term. The analogue of~\eqref{eq:Aufbau} is referred to as the {\em Aufbau} principle in the quantum chemistry literature, while the analogue of \eqref{eq:no_unfilled_shell} is called the {\em no-unfilled shell property}~\cite{DMET2012}. In Hartree--Fock theory, these two properties originate from the fact that the exchange term perfectly compensates the excess of Hartree energy due to self-interaction~\cite{szabo2012modern}.

\begin{proof}[Proof of Theorem~\ref{thm:Aufbau}]

We infer from Proposition \ref{Prop:commutation1storder} that $G(P_\star)$ and $P_\star$ commute. They are therefore simultaneously diagonalizable, i.e.,
\begin{equation}
        P_\star = \sum_{k= 1}^{m} x_{k} x_{k}^T, \quad 
            G(P_\star) =  \sum_{k=1}^{M} \mu_{k} x_{k} x_{k}^T, \quad 
            x_{k}^T x_{l}  = \delta_{k,l}, \quad \mu_k \in \R, \quad  1\leq k,l\leq M.
\end{equation}
Up to reordering, we can assume that $\mu_1 \le \cdots \le \mu_m$ and $\mu_{m+1} \le \cdots \le \mu_M$. Let 
\begin{equation}
    P := \sum_{i= 1}^{m-1} x_{i} x_{i}^T + x_{m+1}x_{m+1}^T = P_\star + x_{m+1}x_{m+1}^T - x_{m}x_{m}^T.
\end{equation}
Clearly, $P \in \mathcal M$ and
\begin{equation*}
    \begin{split}
        J(P)  - J(P_\star) & =  {\rm Tr}\left(G(P_\star) (x_{m+1}x_{m+1}^T - x_{m}x_{m}^T)\right)  \\
        & \quad - \frac{1}{2} {\rm Tr}\left( A(x_{m+1}x_{m+1}^T - x_{m}x_{m}^T)A(x_{m+1}x_{m+1}^T - x_{m}x_{m}^T)\right) \\
        & = (\mu_{m+1} - \mu_{m}) -  \frac{1}{2} {\rm Tr}\left(A^{\frac{1}{2}}(x_{m+1}x_{m+1}^T - x_{m}x_{m}^T)A^{\frac{1}{2}}A^{\frac{1}{2}}(x_{m+1}x_{m+1}^T - x_{m}x_{m}^T)A^{\frac{1}{2}}\right)\\
        & = (\mu_{m+1} - \mu_{m}) -  \frac{1}{2} \lVert S_{m,m+1} \rVert^2,
    \end{split}
\end{equation*}
with $S_{m,m+1} := A^{\frac{1}{2}}(x_{m+1}x_{m+1}^T - x_{m}x_{m}^T)A^{\frac{1}{2}} \in \R^{M \times M}_{\rm sym}$. As $P_\star$ is a minimizer of $J$ on $\mathcal M$, the right-hand side is nonnegative and therefore
$$
\mu_{m+1} - \mu_{m} \ge \frac{1}{2} \lVert S_{m,m+1} \rVert^2 \ge 0,
$$
so that $\mu_1 \le \mu_2 \le  \cdots \le \mu_M$. Renaming these eigenvalues $\lambda_1 \le \lambda_2 \le \cdots \le \lambda_M$, we  obtain~\eqref{eq:Aufbau}. In particular,
$$
\lambda_{m+1} - \lambda_{m} \ge \frac{1}{2} \lVert A^{\frac{1}{2}}(x_{m+1}x_{m+1}^T - x_{m}x_{m}^T)A^{\frac{1}{2}} \rVert^2.
$$ 
If $A$ is invertible, then $A^{\frac{1}{2}}(x_{m+1}x_{m+1}^T - x_{m}x_{m}^T)A^{\frac{1}{2}} \neq 0$ and therefore $\lambda_{m+1} > \lambda_m$. The formula ${P_\star=\mathds 1_{(-\infty,\mu]}(G(P_\star))}$ then follows from functional calculus for Hermitian matrices.
\end{proof}

\subsection{Convexification}
\label{sec:convexification}

We begin this section with the following key observation.

\begin{lemma} \label{lem:convexification}
    Let $\widetilde J:\R^{M \times M}_{\rm sym} \to \R$ be the quadratic function defined by
\begin{equation}\label{eq:def_tildeJ}
\widetilde J(D) := \tr(CD) + \frac 14 \| [A,D] \|^2  \quad \mbox{with} \quad C := B- \frac 12 A^2.
\end{equation}
Then, for all $P \in \mathcal M$, $\widetilde J(P)=J(P)$. In particular, $J$ and $\widetilde J$ have the same critical points on $\mathcal M$.
\end{lemma}

\begin{proof} Let $P \in \mathcal M$. We note that
    \begin{align*}
    \| [A,P] \|^2 &= \tr([A,P]^T [A,P])
    = -\tr([A,P]^2) = - 2 \tr(APAP) + 2\tr( A^2P^2),
\end{align*}
and since $P^2=P$, we have
$
{J(P) = \tr(BP) - \frac 12 \tr(APAP) = \tr(CP) + \frac 14 \| [A,P] \|^2 = \widetilde J(P)}
$.
\end{proof}

We note that the function $\widetilde J$ is convex, which motivates considering the relaxed problem of minimizing $\widetilde J$ over the convex hull of $\mathcal M$
$$
\mathcal K:={\rm CH}(\mathcal M)=\{ D \in \R^{M \times M}_{\rm sym} \; | \; 0 \preceq D \preceq I_M, \, \tr(D)=m \}.
$$
\begin{theorem}
\label{thm:convex_pbm}
Let $A \in \R^{M \times M}_{\rm sym}$ be a positive semidefinite matrix, and $C \in \R^{M \times M}_{\rm sym}$.
The set $\mathcal D_\star$ of the (global) minimizers of the convex optimization problem
\begin{equation} \label{eq:convex_opt}
   \widetilde I= \min_{D \in \mathcal K} \widetilde J(D)
\end{equation}
is a non-empty, compact, convex subset of $\mathcal K$, and we have the following properties:
\begin{enumerate}
    \item The function $\R^{M\times M}_{\rm sym} \ni D \mapsto \nabla \widetilde J(D)=C- \frac 12[[A,D],A] \in \R^{M \times M}_{\rm sym}$ takes a constant value $\widetilde G_\star$ on $\mathcal D_\star$.
    \item Let $\mu_1 \le \cdots \le \mu_M$ be the eigenvalues of $\widetilde G_*$, counted with their multiplicities and ranked in non-decreasing order. Moreover, define 
    \begin{equation*}
        \widetilde {\mathcal D}_\star:= \left\{ D \in \mathcal K \; | \;  \mathds 1_{(-\infty,\mu_m)}(\widetilde G_\star) \preceq D \preceq \mathds 1_{(-\infty,\mu_m]}(\widetilde G_\star) \right\}.
    \end{equation*}
    Then
    \begin{equation} \label{eq:DsubsetDtilde}
        {\mathcal D}_\star \subset \widetilde {\mathcal D}_\star,
    \end{equation}
    and 
    \begin{align}
    \forall D_\star \in \mathcal D_\star, \quad \forall \widetilde D_\star \in \widetilde{\mathcal D}_\star, \quad \widetilde D_\star \in \mathcal D_\star \Leftrightarrow \;  [A,D_\star-\widetilde D_\star]=0 \;   \Rightarrow  \;  {\rm Ran}(D_\star-\widetilde D_\star) \subset \mathcal V, \label{eq:charac_D_star}
    \end{align}
    where $\mathcal V \subset \R^M$ is the largest $A$-invariant vector subspace of ${\rm Ker}(\widetilde G_\star-\mu_m)$.
\item If $\mu_m < \mu_{m+1}$, then 
\begin{equation}
    P_\star:= \mathds 1_{(-\infty,\mu_m]}(\widetilde G_\star) \label{eq:P*=1G*}
\end{equation}
is the unique minimizer of \eqref{eq:convex_opt}, i.e., $\mathcal D_\star=\{P_\star\}$, $P_\star$ is also the unique global minimizer of the original problem~\eqref{eq:OptProblem}--\eqref{eq:CostFunctionJ}, and it holds
\begin{equation} \label{eq:coercivity}
    \forall X \in T_{P_\star}\mathcal M, \quad \langle X,{\rm Hess}_{\mathcal M}J(P_\star)[X]\rangle \ge  (\mu_{m+1}-\mu_m) \|X\|^2 + \frac 12\|[A,X]\|^2.
\end{equation}
\end{enumerate}
\end{theorem}

Inequality~\eqref{eq:coercivity} shows that if $\mu_m < \mu_{m+1}$, then the Riemannian Hessian of the smooth function $J|_{\mathcal M}$ is positive definite at $P_\star$, which implies that the convergence rates of the standard first and second-order local optimization algorithms will be optimal in this case.

\begin{proof}[Proof of Theorem~\ref{thm:convex_pbm}]
    The function $\widetilde J$ is convex on $\R^{M \times M}_{\rm sym}$ and $\mathcal K$ is a compact convex subset of $\R^{M \times M}_{\rm sym}$. The set $\mathcal{D}_\star$  of minimizers of problem \eqref{eq:convex_opt} therefore is a non-empty compact convex subset of $\mathcal{K}$. Let $D_\star \in \mathcal{D}_\star$ be a (global) minimizer of~\eqref{eq:convex_opt}. Using convexity, one has, for any $D \in \mathcal{K}$,
\begin{equation}\label{eq:optimal_0}
\forall \theta \in [0,1], \quad \widetilde J((1-\theta) D_\star + \theta D) \geq \widetilde J(D_\star). \end{equation}
For small $\theta$, $\theta \geq 0$, one can expand \eqref{eq:optimal_0} to get
\begin{align*}
    \widetilde J((1-\theta) D_\star + \theta D)& = \widetilde J( D_\star + \theta (D-D_\star))= \widetilde J( D_\star) + \theta \langle\nabla \widetilde J( D_\star) , D-D_\star\rangle + o(\theta),
\end{align*}
yielding the Euler inequality
\begin{equation*}\label{eq:euler_lagrange}
\forall D \in \mathcal{K}, \quad \tr(\widetilde G_\star D) \ge  \tr( \widetilde G_\star D_\star), \quad \mbox{with} \quad \widetilde G_\star := \nabla \widetilde J( D_\star)=C- \frac 12[[A,D_\star],A].
\end{equation*}
Denoting by  $\mu_1 \le \cdots \le \mu_M$ the eigenvalues of $\widetilde G_\star$, we have 
$$
    \min_{D \in \mathcal{K} }\tr(\widetilde G_\star D) = \min_{D \in \mathcal{M}}\tr(\widetilde G_\star D) = \sum_{k = 1}^{m} \mu_{k},
$$
and
$$
D_\star \in \widetilde {\mathcal D}_\star:= \mathop{\rm argmin}_{D \in \mathcal K}\tr(\widetilde G_\star D)=\left\{ D \in \mathcal K \; | \;  \mathds 1_{(-\infty,\mu_m)}(\widetilde G_\star) \preceq D \preceq 1_{(-\infty,\mu_m]}(\widetilde G_\star) \right\}.
$$
Let $\widetilde D_\star$ be another minimizer of~\eqref{eq:convex_opt}. By convexity, we have that
\begin{align}
\forall \theta \in [0,1], \quad \widetilde I &= \widetilde J((1-\theta) D_\star + \theta \widetilde D_\star)  = \widetilde J(D_\star +\theta(\widetilde D_\star- D_\star)) \nonumber \\
&= \widetilde I + \theta \tr(\widetilde G_\star ( \widetilde D_\star-D_\star)) +\frac{\theta^2}4 \lVert [A,\widetilde D_{\star} - D_\star] \rVert^2. \label{eq:sec_order_conv}
\end{align}
This implies $\tr(\widetilde G_\star ( \widetilde D_\star-D_\star))=0$ which yields $\widetilde D_\star \in \widetilde {\mathcal D}_\star$ and therewith \eqref{eq:DsubsetDtilde}. Moreover, \eqref{eq:sec_order_conv} implies $[A,\widetilde D_{\star} - D_\star]=0$, which yields
$$
\nabla \widetilde J(\widetilde D_\star) = C-  \frac 12[[A,\widetilde D_\star],A] =  C-  \frac 12[[A,D_\star],A] = \widetilde G_\star,
$$
proving the first assertion of the theorem.

\medskip

Let $D_\star \in \mathcal D_\star$ and $\widetilde D_\star \in \widetilde{\mathcal D}_\star$. As $\tr(\widetilde G_\star D_\star) = \tr(\widetilde G_\star \widetilde D_\star)$, we deduce from~\eqref{eq:sec_order_conv} that $\widetilde D_\star$ is in $\mathcal D_\star$ if and only if $[A,D_\star-\widetilde D_\star]=0$. This proves the equivalence in~\eqref{eq:charac_D_star}. To prove the implication, we observe that $\widetilde D_\star \in \widetilde{\mathcal D}_\star$ is in ${\mathcal D}_\star$ if and only if ${\rm Ran}(\widetilde D_\star-D_\star) \subset {\rm Ker}(\widetilde G_\star-\mu_m)$ and $[A,\widetilde D_\star-D_\star]=0$, that is if and only if there exists a family $(u_1,\cdots,u_d)$ of vectors in ${\rm Ker}(\widetilde G_*-\mu_m)$ such that
$$
\widetilde D_\star-D_\star = \sum_{j=1}^{d} n_j u_ju_j^T 
$$
with
$$
d \le {\rm dim}({\rm Ker}(\widetilde G_\star-\mu_m)), \quad n_j \in [-1,1] \setminus \{0\}, \quad \sum_{j=1}^d n_j=0, \quad Au_j= \alpha_j u_j, \quad u_j^Tu_k=\delta_{jk}.
$$
This condition implies that ${\rm Ran}(D_\star-\widetilde D_\star) = \mbox{Span}(u_1,\cdots,u_d) \subset \mathcal V$, where $\mathcal V \subset \R^M$ is the largest $A$-invariant vector subspace of ${\rm Ker}(\widetilde G_\star-\mu_m)$. This completes the proof of the second assertion.

\medskip

The first two statements of the third assertion readily follow from~\eqref{eq:DsubsetDtilde}, the fact that if $\mu_m < \mu_{m+1}$, then $\widetilde {\mathcal D}_\star$ is the singleton $\{1_{(-\infty,\mu_m](\widetilde G_\star)}\}$, and Lemma~\ref{lem:convexification}. To prove the third statement, we fix $X \in T_{P_\star}\mathcal M$ and consider a smooth curve $[-1,1] \ni t \mapsto P(t) \in \mathcal M$ such that $P(0)=P_\star$ and $\dot P(0)=X$. As $P_\star$ is a critical point of $J$ on $\mathcal M$, and $\widetilde J=J$ on $\mathcal M$, we have 
\begin{align*}
    J(P(t))= \widetilde J(P(t))  &=\tr(CP(t)) + \frac 14 \|[A,P(t)]\|^2 \\
    &= \widetilde J(P_\star) + \frac{t^2}2 \left( \tr(\widetilde G(P_\star)\ddot P(0))  + \frac 12 \|[A,X]\|^2 \right) + o(t^2).
\end{align*}
 We thus have 
\begin{equation} \label{eq:Hessian_JMP*}
    \langle X,{\rm Hess}_{\mathcal M}J(P_\star)[X]\rangle =\langle X,{\rm Hess}_{\mathcal M}\widetilde J(P_\star)[X]\rangle = \tr(\widetilde G(P_\star)\ddot P(0))  + \frac 12 \|[A,X]\|^2.
\end{equation}
From \eqref{eq:P*=1G*}, and the fact that $P(t)^2=P(t)$ on $[-1,1]$ implies that $\ddot P(0) P_\star+P_\star \ddot P(0)+2X^2 = \ddot P(0)$, we infer that there exists $U\in O(M)$ and $Y \in \R^{m \times (M-m)}$ such that 
\begin{align*}
& P_\star=U\begin{pmatrix} I_m & 0 \\ 0 & 0 \end{pmatrix} U^T, \qquad &&\widetilde G_\star=U {\rm Diag}(\mu_1,\cdots, \mu_M)U^T , \\
& X=U\begin{pmatrix} 0 & Y \\ Y^T & 0 \end{pmatrix} U^T, \qquad &&\ddot P(0) = U\begin{pmatrix} -2YY^T & * \\ * & 2Y^T Y  \end{pmatrix} U^T.
\end{align*}
We thus have
\begin{align*}
    \tr(\widetilde G(P_\star)\ddot P(0)) &= 2\sum_{i=1}^m \sum_{a=1}^{M-m} \left( \mu_{m+a}-\mu_i\right)|Y_{ia}|^2  \ge  2\left( \mu_{m+1}-\mu_m\right) \|Y\|^2 =  \left( \mu_{m+1}-\mu_m\right) \|X\|^2,
\end{align*}
which concludes the proof.
\end{proof}

\medskip

\begin{remark} Consider the spectral decomposition 
$$
A = \sum_{i=1}^p \alpha_i' P_i, \quad \alpha_1' < \alpha_2' < \cdots < \alpha_p', \quad P_i \in \R^{M \times M}_{\rm sym}, \quad P_iP_j=\delta_{ij} P_i, \quad \sum_{i=1}^p P_i = I_{\R^M},
$$
of $A$ (i.e., $\alpha_1', \cdots, \alpha_p'$ are the {\em distinct} eigenvalues of $A$). Setting $\mathcal V_i:={\rm Ran}(P_i) \cap {\rm Ker}(\widetilde G_\star-\mu_m)$ and $d_i:={\rm dim}(\mathcal V_i)$, we have
$$
\mathcal V=\bigoplus_{1 \le i \le p \; : \; d_i \ge 1} \mathcal V_i.
$$
\end{remark}

\section{Special cases}
\label{sec:special}

In this section, we first investigate the case when $A$ and $B$ commute (Section~\ref{sec:ABcommute}), or ``almost'' commute, in the sense that the commutator of those matrices is small. In the latter case, an approximate solution can be obtained by perturbation theory (Section~\ref{sec:perturbation}). We finally focus in Section~\ref{sec:DMET} on the case which originally motivated our study of the optimization problem~\eqref{eq:OptProblem}--\eqref{eq:CostFunctionJ}, that is the case when $A$ and $B$ are constructed from sub-blocks of some matrix $\gamma \in \R^{L \times L}_{\rm sym}$ such that $0 \preceq \gamma \preceq 1$.

\subsection{Explicit solution when {\it A} and {\it B} commute}
\label{sec:ABcommute}

In the case when $A$ and $B$ commute, problem~\eqref{eq:OptProblem}--\eqref{eq:CostFunctionJ} reduces to the diagonalization of a symmetric matrix.

\begin{proposition}[Explicit solutions when $A$ and $B$ commute]
\label{prop:AB_commute}
Let $A ,B \in \R^{M \times M}_{\rm sym}$ with $A \succeq 0$ and $[A,B]=0$.  As above, we set 
$$
C := B - \frac 12 A^2,
$$
and denote by $c_1 \le c_2 \le \cdots \le c_M$ the eigenvalues of $C$ ranked in non-decreasing order. The minimizers of~\eqref{eq:OptProblem}--\eqref{eq:CostFunctionJ} are given by
\begin{equation}\label{eq:Pstar_commuting_case}
      P_\star = \mathbf{1}_{(-\infty,c_m)}\!\left( C\right) +\Delta, \quad \Delta^2 = \Delta=\Delta^T, \quad \mathrm{Ran}(\Delta) \subset \mathrm{Ker}\left( C-c_m\right), \quad \tr(P_\star)=m, \quad [A,\Delta]=0.
\end{equation}
In particular, if $c_m < c_{m+1}$, $P_\star:=\1_{(-\infty,c_m]}(C)$ is the unique minimizer of~\eqref{eq:OptProblem}--\eqref{eq:CostFunctionJ} and the matrices $A$, $B$, $C$, and $P_\star$ pairwise commute.
\end{proposition}

\begin{proof} Problem~\eqref{eq:OptProblem}--\eqref{eq:CostFunctionJ} is invariant under the action of the orthogonal group $O(M)$ in the sense that if $P_\star$ is a minimizer of $J$ on $\mathcal M$ and $U \in O(M)$, then $Q_\star:=UP_\star U^T$ and is a minimizer of $J_U$ on $\mathcal M$ where
$$
 J_U(Q) := \text{Tr}((UBU^T)Q)- \frac{1}{2} \text{Tr}((UAU^T) Q(UAU^T) Q) .
$$
Since $A$ and $B$ commute, 
there exists $U \in O(M)$ such that
$$
UAU^T = \text{diag}(a_{1}, \dots, a_{M}), \quad UBU^T = \text{diag}(b_{1}, \dots, b_{M}),
$$
and thus 
$$
 UCU^T = \text{diag}(c_{1}, \dots, c_{M}) \quad \mbox{with} \quad c_j = b_j-\frac{a_j^2}2.
$$
After a suitable reordering of the columns of $U$, we may assume without loss of generality that $c_1 \le c_2 \le \cdots \le c_M$. Since $Q_{ij}=Q_{ji}$ and $\sum_{j=1}^M Q_{ij}^2 = Q_{ii}$ for all $Q \in \mathcal M$, we thus have for this choice of~$U$, 
\begin{align*}
         J_U(Q) & = \sum_{i = 1}^{M} b_i Q_{ii} - \frac{1}{2} \sum_{i,j=1}^{M} a_{i} a_{j} Q_{ij}^2  = \sum_{i = 1}^{M} b_i Q_{ii} + \frac{1}{4} \sum_{i,j=1}^{M}( (a_i -a_j)^2 - a_{i}^2 - a_{j}^2) Q_{ij}^2 \\
         &= \sum_{i = 1}^{M} b_i Q_{ii} -  \frac 12 \sum_{i=1}^M a_{i}^2 \sum_{j=1}^M Q_{ij}^2  + \frac{1}{4} \sum_{i,j=1}^M  (a_i -a_j)^2 Q_{ij}^2  = \sum_{i = 1}^{M} c_iQ_{ii} + \frac{1}{4} \sum_{i,j=1}^{M} (a_i -a_j)^2Q_{ij}^2.
         \end{align*}
Consequently, the cost function $J_U$ reaches a minimum if and only if $Q_{ij} = 0$ whenever $a_i \neq a_j$, and 
\begin{equation*}
Q_{ii} = 
\left\lbrace 
\begin{aligned}
1,~{\rm if}~c_i < c_m,  \\
0,~{\rm if}~c_i > c_m.
\end{aligned}
\right.
\end{equation*}
Since $Q^2=Q=Q^T$, the minimizers $Q_\star$ of $J_U$ on $\mathcal M$ are given by
\begin{equation} \label{eq:decomp_Pstar}
      Q_\star =  \left( \begin{array}{ccc} I_k & 0 & 0 \\ 0 & 0 & 0 \\ 0 & 0 & 0 \end{array} \right) + \Delta_\star \quad \mbox{with} \quad \Delta_\star =  \left( \begin{array}{ccc} 0 & 0 & 0 \\ 0 & \delta & 0 \\ 0 & 0 & 0 \end{array} \right) \end{equation}
      where
      $$
      \quad k:=\max\{j \; : \; c_j < c_m\}, \quad \delta^2 = \delta=\delta^T, \quad \tr(\delta)=m-k, \quad [\delta,\text{diag}(a_{k+1}, \dots, a_{k+p})]=0,
      $$
       $p$ denoting the multiplicity of the eigenvalue $c_m$, with the convention that $k=0$ if no $c_j$ is lower than $c_m$.
Multiplying $Q_\star$ by $U^T$ on the left and $U$ on the right, we obtain~\eqref{eq:Pstar_commuting_case}. 

\medskip

\noindent
Finally, if $c_m < c_{m+1}$, we have that $k=m-p$, $\delta = I_p$, $\Delta_\star=\mathds 1_{\{c_m\}}(C)$, and therefore $P_\star:=\1_{(-\infty,c_m]}(C)$. After multiplying $Q_\star$ by $U^T$ on the left and $U$ on the right, we obtain that the matrices $A$, $B$, $C$ and $P_\star$ pairwise commute.
\end{proof}

\subsection{Perturbation analysis}
\label{sec:perturbation}

We now investigate the case when $[A,B]$ is ``small''. To that end, we consider two commuting matrices $A_0$ and $B_0$ in $\R^{M \times M}_{\rm sym}$ with $A_0 \succeq 0$, and two $\R^{M \times M}_{\rm sym}$-valued functions $\epsilon \mapsto A^\epsilon$ and  $\epsilon \mapsto B^\epsilon$ real-analytic in a small neighborhood of $\epsilon=0$ in $\R$ and such that $A^0=A_0$ and $B^0=B_0$. We then consider the family of Grassmann optimization problems
\begin{equation} \label{eq:pertubCost_comm}
I^\epsilon := \min_{P \in \mathcal M} J^\epsilon(P) \quad \mbox{with} \quad J^\epsilon(P):= \tr(B^\epsilon P) - \frac 12 \tr(A^\epsilon P A^\epsilon P).
\end{equation}
Let $c_1^0 \le c_2^0 \le \cdots \le c_M^0$ be the eigenvalues of $C_0:=B_0-\frac 12 (A_0)^2$, counting multiplicities. Under the condition that $c_m^0 < c_{m+1}^0$,
we know from Proposition~\ref{prop:AB_commute} that for $\epsilon=0$, $P_0:=\1_{(-\infty,c_m^0]}(C_0)$ is the unique global minimizer of~\eqref{eq:pertubCost_comm} and that the matrices $A_0$, $B_0$, $C_0$ and $P_0$ pairwise commute. Let $U_0 \in O(M)$ be such that
\begin{align*}
& A_0 = U_0\text{diag}(a_{1}^0, \dots, a_{M}^0) U_0^T, \quad B_0 = U_0\text{diag}(b_{1}^0, \dots, b_{M}^0) U_0^T, \\
 & C_0 = U_0\text{diag}(c_{1}^0, \dots, c_{M}^0)U_0^T \quad \mbox{with} \quad c_j^0 = b_j^0-\frac{(a_j^0)^2}2, \quad  \mbox{and} \quad P_0 = U_0 \begin{pmatrix} I_m & 0 \\ 0 & 0 \end{pmatrix} U_0^T.
\end{align*}
Expanding $A^\epsilon$, $B^\epsilon$, and $C^\epsilon := B^\epsilon - \frac 12 (A^\epsilon)^2$ as
$$
A^\epsilon = \sum_{n=0}^{+\infty} \epsilon^n A_n, \quad B^\epsilon = \sum_{n=0}^{+\infty} \epsilon^n B_n, \quad \mbox{and} \quad C^\epsilon = \sum_{n=0}^{+\infty} \epsilon^n C_n, \quad A_n,B_n,C_n \in \R^{M \times M}_{\rm sym}, 
$$
and using the decompositions 
$$
A_1= U_0 \begin{pmatrix}
             A_1^{\rm oo} &  A_1^{\rm ov} \\ A_{1}^{\rm vo} & A_1^{\rm vv}
         \end{pmatrix} U_0^T \quad \mbox{and} \quad B_1= U_0 \begin{pmatrix}
             B_1^{\rm oo} &  B_1^{\rm ov} \\ B_{1}^{\rm vo} & B_1^{\rm vv}
         \end{pmatrix} U_0^T,
         $$
we have the following result.

\medskip

\begin{proposition} If $c_m^0 < c_{m+1}^0$, there exists $\epsilon_0 > 0$ such that for all $\epsilon \in (-\epsilon_0,\epsilon_0)$, problem~\eqref{eq:pertubCost_comm} has a unique global minimizer $P^\epsilon$, and the function $(-\epsilon_0,\epsilon_0) \ni \epsilon \mapsto P^\epsilon \in \R^{M \times M}_{\rm sym}$ is real-analytic. We have for $|\epsilon|$ small enough
\begin{equation} \label{eq:expantion_Pepsilon}
P^\epsilon = \sum_{n=0}^{+\infty} \epsilon^n P_n \quad \mbox{with} \quad P_0= U_0  \begin{pmatrix}
             I_m & 0 \\ 0 & 0
         \end{pmatrix} U_0^T \quad \mbox{and} \quad P_1=  U_0 \begin{pmatrix}
             0 & P_{1}^{\rm ov}\\ P_{1}^{\rm vo} & 0
         \end{pmatrix} U_0^T,
\end{equation}
with 
\begin{equation}\label{eq:P1ov}
\forall 1 \le i \le m, \quad 1 \le j \le M-m, \quad (P_{1}^{\rm ov})_{ij} = -\frac{(B_{1}^{\rm ov})_{ij} - a^0_{i} (A_{1}^{\rm ov})_{ij}}{c_{m+j}^0-c_i^0 +\frac 12 (a_{m+j}^0-a_i^0)^2}.
\end{equation}
 \end{proposition}
 
\begin{proof} The manifold $\mathcal M$ is compact and $J^\epsilon\in C^\infty(\mathcal M)$,
hence $J^\epsilon$ admits at least one global minimizer $P^\epsilon$ on $\mathcal M$, which is also a global minimizer of $\widetilde J^\epsilon$ on $\mathcal M$, where
$$
\forall P \in \R^{M \times M}_{\rm sym}, \quad \widetilde J^\epsilon(P) := \tr(C^\epsilon P) + \frac 14 \|[A^\epsilon,P]\|^2,
$$
since $\widetilde J^\epsilon$ and $J^\epsilon$ agree on $\mathcal M$. As $\epsilon \mapsto A^\epsilon$  and $\epsilon \mapsto C^\epsilon$ are real-analytic on some interval $(-\widetilde \epsilon_0,\widetilde\epsilon_0)$ and $\mathcal M$ is compact, there exists $c \in \R_+$ such that for all $\epsilon \in (-\widetilde\epsilon_0,\widetilde\epsilon_0)$, we have on the one hand
\begin{align*}
\widetilde J^\epsilon(P_0) & \ge \widetilde J^\epsilon(P^\epsilon) \ge  \tr(C^\epsilon P^\epsilon) \ge \tr(C^0 P^\epsilon) - c |\epsilon|  \ge \widetilde J^0(P_0)+ \frac 12 (c^0_{m+1}-c^0_m) \|P^\epsilon-P_0\|^2 - c|\epsilon|,
\end{align*}
and on the other hand
$$
\widetilde J^\epsilon(P_0) \le \widetilde J^0(P_0) + c |\epsilon|.
$$
It follows that 
\begin{equation} \label{eq:Pepsilon-P0}
\forall \epsilon \in (-\widetilde\epsilon_0,\widetilde\epsilon_0), \quad \|P^\epsilon-P_0\| \le 2 \sqrt{\frac{c |\epsilon|}{c^0_{m+1}-c^0_m}}.
\end{equation}
Let us introduce the function 
\begin{equation}
F: (-\widetilde\epsilon_0,\widetilde\epsilon_0) \times \mathcal M \ni (\epsilon,P) \mapsto F(\epsilon,P):=[P,[P,B^\epsilon-A^\epsilon P A^\epsilon]] \in \mathcal T\mathcal M,
\end{equation}
which satisfies $F(\epsilon,P) \in \mathcal T_P\mathcal M$ for all  $(\epsilon,P)  \in (-\widetilde\epsilon_0,\widetilde\epsilon_0) \times \mathcal M$. The function $F$ is real-analytic and we have $F(0,P_0)=0$ and 
$$
d_{(0,P_0)} F|_{\mathcal T_{P_0}\mathcal M} \to \mathcal T_{0}(\mathcal T_{P_0}\mathcal M) \equiv \mathcal T_{P_0}(\mathcal M) 
$$
is invertible. Indeed, we have that
\begin{align}
&\forall Q = U_0 \begin{pmatrix} 0 & X_{\rm ov} \\ X_{\rm vo} & 0 \end{pmatrix} U_0^T \in \mathcal T_{P_0}\mathcal M, \nonumber \\
& d_{(0,P_0)} F|_{\mathcal T_{P_0}\mathcal M} (Q) = [P_0,[Q,B_0-A_0P_0A_0]]-[P_0,[P_0,A_0QA_0]] = U_0 \begin{pmatrix} 0 & Y_{\rm ov} \\ Y_{\rm vo} & 0 \end{pmatrix} U_0^T, \label{eq:IFT_cond}
\end{align}
with
\begin{equation}\label{eq:IFT_cond2}
\forall 1 \le i \le m, \quad \forall 1 \le j \le M-m, \quad [Y_{\rm ov}]_{ij} = \left( \underbrace{c_{m+j}^0-c_i^0 +\frac 12 (a_{m+j}^0-a_i^0)^2}_{> 0} \right) [X_{\rm ov}]_{ij} .
\end{equation}
It follows from~\eqref{eq:Pepsilon-P0} and the real-analytic version of the implicit function theorem for parameter-dependent vector fields on analytic manifolds that there exists $\epsilon_0 > 0$ such that for all $\epsilon \in (-\epsilon_0,\epsilon_0)$, $J^\varepsilon$ has a unique global minimizer $P^\epsilon$ on $\mathcal M$, that the function $\epsilon \mapsto P^\epsilon$ is real analytic on $(-\epsilon_0,\epsilon_0)$ and that the matrix $P_1$ in \eqref{eq:expantion_Pepsilon} satisfies
$$
 [P_0,[P_1,B_0-A_0P_0A_0]]+[P_0,[P_0,B_0-A_0P_1A_0]] = - [P_0,[P_0,B_1-A_1P_0A_0-A_0P_0A_1]].
$$
In view of~\eqref{eq:IFT_cond2}, we get~\eqref{eq:P1ov}.
\end{proof}

\subsection{DMET bath-construction problem}
\label{sec:DMET}

Density Matrix Embedding Theory (DMET), introduced by Knizia and Chan, is a static quantum embedding method based on the idea that the entanglement between a local fragment and its environment is effectively low rank. This observation leads to a singular value decomposition of a reference state, producing a minimal bath that effectively captures the fragment-environment interaction and reduces the full many-body problem to a self-consistent finite-dimensional embedding~\cite{DMET2012,DMET2013,bulik2014density,kretchmer2018real,wouters2016practical}. 
Several works have modified the original DMET fixed-point formulation, reformulating or replacing the correlation-potential optimization to improve numerical robustness and expose the structure of the embedding map~\cite{wu2019projected,wu2020enhancing,lin2022variational,faulstich2022pure}. These variants define alternative nonlinear projection maps on reduced density matrices whose stability, convexity properties, and convergence behavior differ in essential ways from the original scheme. 
The DMET bath construction can be viewed not merely as a SVD-based compression, but as a structured unitary transformation acting on the one-particle Hilbert space~\cite{negre2025DMET}. In particular, Householder and block-Householder transformations recast bath generation as a parametrized unitary rotation, enabling ensemble or multi-state embeddings and allowing additional optimality or physically motivated constraints to be imposed on the bath orbitals~\cite{Sekaran2021,yalouz2022quantum,Sekaran2022,Marecat2023,cernatic2024fragment}.
Beyond the standard one-body DMET matching principle, several extensions enrich the information content of the embedding to better capture fragment-environment correlations. Representative directions include incorporating higher-order fluctuations and screening physics~\cite{booth2015spectral,fertitta2018rigorous,Fertitta2019,sriluckshmy2021fully,scott2021extending,nusspickel2022systematic} and introducing correlated bath wavefunctions~\cite{tsuchimochi2015density}.The purpose of this section is to explain how the bath-construction problem encountered in the extension of DMET introduced in~\cite{negre2025DMET} reduces to problem~\eqref{eq:OptProblem}--\eqref{eq:CostFunctionJ}.\\

Consider an $N$-body fermionic quantum system with associated one-particle state space $\mathcal{H}$. For simplicity, we assume that $\mathcal H$ is of finite dimension $L$. This is the case in practice in quantum chemistry since the electronic Hamiltonian is discretized in a finite basis of atomic orbitals (or other kinds of basis functions), as well as for lattice models with a finite number of sites. We also assume that the system is non-magnetic and non-relativistic, which allows us to work with a {\rm real} Hilbert space $\mathcal H$. The most common problem in quantum chemistry is to compute the properties of the ground state of the electronic Hamiltonian $H_N$ of the system under study. In mathematical terms, this amounts to computing quantities of interest (e.g., electronic density, interatomic forces, etc.) derived from the normalized eigenfunction $\Psi_N^0$ of $H_N$ associated with its lowest eigenvalue $E_N^0$, referred to as the ground-state wave function and ground-state energy, respectively. 
The major difficulty is that $H_N$ is a self-adjoint operator acting on the fermionic $N$-particle state space $\mathcal H_N:=\bigwedge^N \mathcal H$, the antisymmetrized tensor product of $N$ copies of $\mathcal H$. 
Since ${\mathrm{dim}}(\mathcal H_N) = {L \choose N}$, computing the $N$-electron ground-state energy and wave function amounts to finding the lowest eigenpair of an ${L \choose N} \times {L \choose N}$ real symmetric matrix. 
The electronic problem therefore suffers from the curse of dimensionality, making it numerically intractable for large values of $L$ and $N$.

\medskip

DMET is an iterative method for computing approximations of some entries of the ground-state one-particle reduced density matrix (1-RDM), a key quantity in quantum physics. DMET can be interpreted as a domain decomposition method for the quantum many-body problem (see~\cite{CFKLL2025,negre2025DMET} for more details). Its works as follows. First, the system is partitioned into $N_{\rm f}$ non-overlapping fragments on the basis of physical arguments, leading to an orthogonal decomposition of $\mathcal H$, i.e.
\begin{equation}\label{eq:fragmentation}
\mathcal H = \mathcal H_{{\rm frag},1} \oplus \cdots \oplus \mathcal H_{{\rm frag},N_{\rm f}}, \quad {\rm dim}(\mathcal H_{{\rm frag},x})=\ell_x, \quad \sum_{x=1}^{N_{\rm f}} \ell_x=L.
\end{equation}
At each DMET iteration, the state of the whole system is represented by a coarse descriptor. In conventional DMET, this coarse descriptor is a Slater-type 1-RDM, i.e. a rank-$N$ orthogonal projector corresponding to a unique $N$-body state represented in $\mathcal H_N$ by a Slater determinant. The set of such 1-RDMs is thus
\begin{align*}
{\Gamma}_{\rm Slater}:&=\{ \gamma \in {\mathcal S}(\mathcal H) \; | \; \gamma^2=\gamma, \; \tr(\gamma)=N\},
\end{align*}
where ${\mathcal S}(\mathcal H)$ is the space of symmetric operators on $\mathcal H$. In the extensions of DMET introduced in~\cite{negre2025DMET}, the set of coarse descriptors is the larger set 
$$
{\Gamma}_{\rm MS}:=\{ \gamma \in {\mathcal S}(\mathcal H) \; | \; 0 \preceq \gamma \preceq 1, \; \tr(\gamma)=N\}
$$
of all $N$-representable mixed-state 1-RDMs. The condition $ 0 \preceq \gamma \preceq 1$ corresponds to the Pauli principle~\cite{pauli1925zusammenhang}. From a geometrical point of view, ${\Gamma}_{\rm Slater}={\rm Gr}(N,\mathcal H)$, where, for any real vector space~$\mathcal X$ and integer $1 \le n \le {\rm dim}(\mathcal X)-1$,
$$
{\rm Gr}(n,\mathcal X):= \left\{ M \in \mathcal S(\mathcal X) \; | \; M^2=M, \; \tr(M)=n \right\}
$$
is the real Grassmann manifold of the rank-$n$ orthogonal projectors in $\mathcal S(\mathcal X)$, and ${\Gamma}_{\rm MS}$ is the convex hull of ${\Gamma}_{\rm Slater}$. 

\medskip

At each DMET iteration, we have at hand a trial global descriptor $\gamma$. From $\gamma$, we construct for each fragment $x$, a cluster subspace $ \mathcal H_{{\rm clust},x}^\gamma$ satisfying $\mathcal H_{{\rm frag},x} \subset \mathcal H_{{\rm clust},x}^\gamma \subset \mathcal H$ of dimension $(\ell_x+m_x)$ with $1 \le m_x \ll L$, which we decompose as the orthogonal sum of the fragment subspace $\mathcal H_{{\rm frag},x}$ and the so-called {\rm bath} subspace $\mathcal H_{{\rm bath},x}^\gamma$:
$$
\mathcal H_{{\rm clust},x}^\gamma = \mathcal H_{{\rm frag},x} \oplus \mathcal H_{{\rm bath},x}^\gamma, \quad \mbox{with} \quad \dim(\mathcal H_{{\rm bath},x}^\gamma)=m_x.
$$
As detailed below, our method for constructing $\mathcal H_{{\rm clust},x}^\gamma$ is to minimize in some sense (vide infra) the entanglement between  $\mathcal H_{{\rm clust},x}^\gamma$ and its orthogonal complement $\mathcal{H}_{{\rm env},x}^\gamma := (\mathcal{H}_{{\rm clust},x}^\gamma)^\perp$.
The many-body Hamiltonian $H_N$ of the whole system is then written in second quantization and projected on the Fock space generated by $\mathcal H_{{\rm clust},x}^\gamma$. This leads to a quantum many-body problem of dimension $2^{\ell_x+m_x} \ll {L \choose N}$, which can be solved either by exact diagonalization if $2^{\ell_x+m_x}$ is not too large, or by advanced approximation methods. The solutions of the $N_{\rm f}$ cluster subproblems are then recombined to form the trial global descriptor for the next iteration, and the process is iterated until convergence. This elementary DMET procedure is commonly combined with relaxation and extrapolation methods to achieve or accelerate convergence.

\medskip

We now focus on the bath-construction subproblem: Given 
\begin{itemize}
    \item[(i)] a fragment subspace $\mathcal H_{\rm f}$ of dimension $\ell$,\vspace{-2mm}
    \item[(ii)] a bath dimension $m$,\vspace{-2mm}
    \item[(iii)] a trial  1-RDM $\gamma$ in ${\Gamma}_{\rm Slater}$ or in ${\Gamma}_{\rm MS}$,
\end{itemize}
construct a bath subspace $\mathcal H_{{\rm bath}}^\gamma \subseteq \mathcal H$ of dimension $m$ that is orthogonal to $\mathcal H_{\rm frag}$. Equivalently, construct a cluster subspace $\mathcal H_{{\rm clust}}^\gamma \subseteq \mathcal H$ of dimension $(\ell+m)$ containing $\mathcal H_{\rm frag}$.

Without loss of generality, we can identify ${\mathcal H}$ with $\R^L$  and assume that $\mathcal{H}_{\rm frag} = \text{span}(e_1, \dots, e_\ell)$, where $(e_1, \dots e_L)$ is the canonical basis of $\R^L$. Then, the orthogonal projectors $\Pi_{\rm frag}$ and $\Pi_{\rm clust}^\gamma$ onto $\mathcal{H}_{\rm frag}$ and $\mathcal H_{{\rm clust}}^\gamma$, respectively, are represented by block matrices of the form
\begin{equation}\label{eq:PfPc}
\Pi_{\rm frag}=\left( \begin{array}{cc} I_\ell & 0  \\ 0 & 0 \end{array} \right), \quad 
\Pi_{\rm clust}^\gamma=\left( \begin{array}{cc} I_{\ell} & 0  \\ 0 & P_{\rm clust}^{\gamma} \end{array} \right), \quad \mbox{with} \quad P_{\rm clust}^{\gamma} \in {\rm Gr}(m,\R^M), \quad M=L-\ell.
\end{equation}
The idea is to choose $\mathcal H_{{\rm clust}}^\gamma$ in such a way that the 1-RDM $\gamma$ is as disentangled as possible with respect to the cluster-environment partition of the system. Note that if $\mathcal H=\mathcal H_{\rm S_1} \oplus \mathcal H_{\rm S_2}$ defines a partition of the system into two subsystems $\rm S_1$ and $\rm S_2$, it is not possible to determine whether a quantum state is disentangled with respect to the $\rm S_1$-$\rm S_2$ partition from its 1-RDM alone. On the other hand, if a quantum state with 1-RDM $\gamma$ is disentangled w.r.t. the $\rm S_1$-$\rm S_2$ partition, then $\Pi_{{\rm S_1}} \gamma \Pi_{{\rm S_2}}=0$, where $\Pi_{\rm S_1}$ and $\Pi_{\rm S_2}$ are the orthogonal projectors onto $\mathcal H_{\rm S_1}$ and $\mathcal H_{\rm S_2}$, respectively. We will therefore use the quantity $\Pi_{{\rm S_1}} \gamma \Pi_{{\rm S_2}}$ as a proxy of entanglement, and define $\mathcal H_{{\rm clust}}^\gamma$ as the range of an orthogonal projector solving
\begin{align} \label{eq:opt_Pi}
\Pi_{\rm clust}^\gamma \in \mathop{\rm argmin}_{
\substack{\Pi \in {\rm Gr}(\ell+m,\R^L) \\ \Pi\,\Pi_{\rm frag}=\Pi_{\rm frag}}
}  \mathcal J^\gamma(\Pi), \quad \mbox{with} \quad \mathcal J^\gamma(\Pi):=\|\Pi\gamma\Pi^\perp\|^2, \quad \Pi^\perp=I_L-\Pi.
\end{align}
Note that the condition $\Pi\,\Pi_{\rm frag}=\Pi_{\rm frag}$ is equivalent to $\Pi$ being of the form 
\begin{equation}\label{eq:Pi_P}
\Pi=\left( \begin{array}{cc} I_{\ell} & 0  \\ 0 & P \end{array} \right) \quad \mbox{with} \quad  P \in {\rm Gr}(m,\R^M).
\end{equation}
Decomposing $\gamma$ as
\begin{align}
\label{rdm1_basis1}
     \gamma = \begin{pmatrix}
        \gamma_{\rm frag } & \gamma_{{\rm frag}, \rm ext}\\
        \gamma_{{\rm ext}, \rm frag} & \gamma_{\rm ext} \end{pmatrix}, \quad \gamma_{\rm frag } \in \R^{\ell \times \ell}_{\rm sym}, \quad   \gamma_{\rm ext} \in \R^{M \times M}_{\rm sym}, \quad \gamma_{{\rm frag}, \rm ext} = \gamma_{{\rm ext}, \rm frag}^T \in \mathbb{R}^{\ell \times (L-\ell)},
\end{align}
an elementary calculation shows that for any $\Pi$ of the form \eqref{eq:Pi_P}, it holds
\begin{align}
\mathcal J^\gamma(\Pi) &= \left\| \left( \begin{array}{cc} I_{\ell} & 0  \\ 0 & P \end{array} \right) \begin{pmatrix}
        \gamma_{\rm frag } & \gamma_{{\rm frag}, \rm ext}\\
        \gamma_{{\rm ext}, \rm frag} & \gamma_{\rm ext} \end{pmatrix} \left( \begin{array}{cc} 0 & 0  \\ 0 & (I_M-P) \end{array} \right) \right\|^2 \nonumber  \\
        &= \tr(\gamma_{{\rm frag}, \rm ext}(1-P) \gamma_{{\rm ext},{\rm frag}}) + \tr\left( P \gamma_{\rm ext } (1-P) \gamma_{\rm ext } \right) \nonumber \\
        &= 2 J(P) + \|\gamma_{{\rm ext}, \rm frag}\|^2, \label{eq:J_gamma}
\end{align}
with
\begin{equation}\label{eq:BCP-AB}
J(P) = \tr(BP)-\frac 12 \tr(APAP), \quad  A:= \gamma_{\rm ext }, \quad \mbox{and} \quad B:=\frac 12 \left(\gamma_{\rm ext }^2- \gamma_{{\rm ext}, \rm frag} \gamma_{{\rm frag}, \rm ext}  \right).
\end{equation}
The matrix $C$ in~\eqref{eq:def_tildeJ} is then given by 
\begin{equation}\label{eq:BCP-C}
C:=B-\frac 12 A^2 = - \frac 12 \gamma_{{\rm ext}, \rm frag} \gamma_{{\rm frag}, \rm ext} .
\end{equation}
Since $\gamma \succeq 0$, we have $A \succeq 0$.
Problem~\eqref{eq:opt_Pi} is therefore equivalent to~\eqref{eq:OptProblem}--\eqref{eq:CostFunctionJ} with the above choices of matrices $A$ and $B$.
Note that we chose the Frobenius norm to measure the magnitude of $\Pi\gamma\Pi^\perp$. Other choices, for instance involving the kinetic energy of the 1-RDM $\gamma$, could also be considered, leading to different Grassmann optimization problems. 

\medskip

In this work, full disentanglement is said to be achieved if
\begin{equation}\label{eq:full_disentanglement}
\mathop{\rm min}_{
\substack{\Pi \in {\rm Gr}(\ell+m,\R^L) \\ \Pi\,\Pi_{\rm frag}=\Pi_{\rm frag}}
}  \mathcal J^\gamma(\Pi) = 0 \quad \mbox{(1-RDM full disentanglement criteria)}.
\end{equation}
Equivalently, 1-RDM full disentanglement is achieved if there exists a cluster subspace $\mathcal H_{{\rm clust}}^\gamma$ for which $\gamma$ is block-diagonal in the decomposition 
$$
\mathcal H = \mathcal H_{{\rm clust}}^\gamma \oplus {\mathcal H_{{\rm clust}}^\gamma}^\perp.
$$
In other words, 1-RDM full disentanglement can be achieved if and only if we can find an $(\ell+m)$-dimensional $\gamma$-invariant subspace of $\mathcal H$ containing $\mathcal H_{\rm frag}$. 

\medskip

In conventional DMET, $\gamma \in \Gamma_{\mathrm{Slater}}$. In this case, the following result holds.

\begin{proposition}
\label{prop:DMET_proj_sol}
    If $\gamma \in \Gamma_{\mathrm{Slater}}$, $m\geq\ell$ and $\ell+m \le L$, then 1-RDM full disentanglement can be achieved and constructing an admissible $\mathcal H_{{\rm clust}}^\gamma$ reduces to performing a pivoted LQ decomposition of the matrix $\gamma_{\rm frag,ext}$.
    \end{proposition}

\begin{proof} Since $\gamma^2=\gamma$, $\mathcal H_{\rm frag} + \gamma \mathcal H_{\rm frag}$ is a $\gamma$-invariant subspace of $\mathcal H$ containing $\mathcal H_{\rm frag}$ of dimension lower or equal to $2\ell$. An orthonormal basis of $\mathcal H_{\rm frag} + \gamma \mathcal H_{\rm frag}$ can be obtained by enriching the basis $(e_1,\cdots,e_\ell)$ by Gram--Schmidt orthonormalization of $(\gamma e_1,\cdots,\gamma e_\ell)$, which amounts to performing an LQ decomposition of $\gamma_{\rm frag,ext}$ (pivoting is needed if the matrix is rank-deficient). Any $(\ell+m)$-dimensional subspace of $\mathcal H$ containing  $\mathcal H_{\rm frag} + \gamma \mathcal H_{\rm frag}$ is an admissible $\mathcal H_{{\rm clust}}^\gamma$.
\end{proof}

Note that $\mathcal H_{{\rm clust}}^\gamma$ is unique if $m=\ell$ and ${\rm dim}(\mathcal H_{\rm frag} + \gamma \mathcal H_{\rm frag})=2\ell$, a non-degeneracy condition invoked in~\cite{CFKLL2025}, and usually satisfied in practice in conventional DMET. 

\medskip

In the proposed extensions of DMET, we only assume $\gamma \in \Gamma_{\mathrm{MS}}$. The following result provides a necessary and sufficient condition for full disentanglement to be achievable.

\begin{proposition}
    Any $\gamma \in \Gamma_{\rm MS}$ can be uniquely decomposed as
    \begin{equation}\label{eq:decompD}
    \begin{split}
         & \gamma =  \sum_{i=1}^s n_i P_i, \quad P_i \in {\rm Gr}(m_i,\R^L), \quad m_i \ge 1, \quad P_iP_j=\delta_{ij}P_i, \quad 1 \le i \le s \le L,\\ 
         &\quad 0 <  n_1 < n_2 < \cdots < n_s \le 1,
          \quad  \sum_{i=1}^s n_im_i=N.     
    \end{split}
    \end{equation}
    The linear space
        $$
    \mathcal X:=\mathcal{H}_{\rm frag} + \left( P_1 \mathcal{H}_{\rm frag} \oplus \cdots \oplus P_s \mathcal{H}_{\rm frag} \right) 
    $$
    is the smallest $\gamma$-invariant subspace of $\R^L$ containing $\mathcal{H}_{\rm frag}$.
    1-RDM full disentanglement can therefore be achieved if and only if $m \ge {\rm dim}(\mathcal X)-\ell$.
\end{proposition}

\begin{proof} The decomposition \eqref{eq:decompD} is the usual spectral decomposition of $\gamma$. Indeed, since $\gamma  \in \Gamma_{\rm MS}$, its spectrum $\sigma(\gamma)$ lies in the range $[0,1]$. Defining $s:=| \sigma(\gamma) \cap (0,1] |$, we have 
$$
\sigma(\gamma) \cap (0,1]=\{n_1, \cdots, n_s\} \quad \mbox{with} \quad 0 <  n_1 < n_2 < \cdots < n_s \le 1,
$$
and 
$$
\gamma =  \sum_{i=1}^s n_i P_i \quad \mbox{with} \quad  P_i:=\1_{\{n_i\}}(\gamma).
$$
Denoting by $m_i$ the rank of $P_i$ (i.e., the multiplicity of the eigenvalue $n_i$), we have 
$$
P_iP_j=\delta_{ij}P_i, \quad 1 \le i,j \le s, \quad \mbox{and} \quad \sum_{i=1}^s n_im_i= \sum_{i=1}^s n_i \tr(P_i)=\tr(\gamma)=N.
$$
Since $\gamma$ is real symmetric, the above decomposition is unique.

\medskip

Since $\gamma  P_i=n_i P_i$, we have that
\begin{align*}
\gamma \mathcal X & = \gamma \mathcal{H}_{\rm frag} + \gamma  P_1 \mathcal{H}_{\rm frag} + \cdots + \gamma  P_s \mathcal{H}_{\rm frag} \\
& = \left( \sum_{i=1}^s n_iP_i \right) \mathcal{H}_{\rm frag}  + P_1 \mathcal{H}_{\rm frag} + \cdots +   P_s \mathcal{H}_{\rm frag}  \subset P_1 \mathcal{H}_{\rm frag} \oplus \cdots \oplus P_s \mathcal{H}_{\rm frag} \subset \mathcal X,
\end{align*}
which proves that $\mathcal X$ is a $\gamma$-invariant subspace of $\R^L$ containing $\mathcal H_{\rm frag}$.

\medskip

Let $\mathcal Y$ be a $\gamma$-invariant subspace of $\R^L$ containing $\mathcal H_{\rm frag}$. As $\gamma \mathcal Y \subset \mathcal Y$ and $\mathcal H_{\rm frag} \subset \mathcal Y$, we have that $f(\gamma) \mathcal H_{\rm frag} \subset \mathcal Y$ for any Borel function $f: \R \to \R$. In particular, we have for all $1 \le i \le s$, $P_i  \mathcal H_{\rm frag} = \1_{\{n_i\}}(\gamma) \mathcal H_{\rm frag}  \subset \mathcal Y$. As we also have $\mathcal H_{\rm frag} \subset \mathcal Y$ by hypothesis, we obtain that $\mathcal X\subset \mathcal Y$, which concludes the proof.

\medskip

Any $(\ell+m)$-dimensional $\gamma$-invariant subspace $\mathcal Y$ of $\mathcal H$ containing $\mathcal H_{\rm frag}$ contains $\mathcal X$, and thus $(\ell+m)={\rm dim}(\mathcal Y) \ge {\rm dim}(\mathcal X)$. Conversely, we have that $\mathcal X^\perp$ is $\gamma$-invariant (as $\gamma$ is symmetric), hence
spanned by eigenvectors of $\gamma$. Assuming that $m \ge {\rm dim}(\mathcal X)-\ell$, we can select $\ell+m-{\rm dim}(\mathcal X)$ of them, consider the vector subspace of $\mathcal H$ spanned by these vectors and the vectors of $\mathcal X$, and thus obtain an $(\ell+m)$-dimensional $\gamma$-invariant subspace containing $\mathcal H_{\rm frag}$.
\end{proof}

\medskip

Note that if $\gamma \in \Gamma_{\rm Slater}$, then
$s = 1$, $n_1 = 1$, and $P_1 = \gamma$.
In this case, 
$
\mathcal X = \mathcal H_{\rm frag}+\gamma \mathcal H_{\rm frag},
$
and $\dim(\mathcal X) \le 2\ell$. In the general case $\gamma \in \Gamma_{\rm MS}$, we have the bound
$$
\dim(\mathcal X)
\le
\ell + \sum_{i=1}^s \min(m_i,\ell).
$$
Consequently, full disentanglement can be achieved whenever
$$
m \ge \sum_{i=1}^s \min(m_i,\ell).
$$

This bound improves upon the previously reported condition
$m \ge s\ell$ in~\cite{Marecat2023,Cernatic2024},
which was derived using arguments based on minimal polynomials and
Householder rotations.
An optimal bound can be obtained by computing
$\dim(\mathcal X)$ via a singular value decomposition.

\section{Numerical methods}
\label{sec:numerical_methods}

We consider three classes of numerical methods for solving~\eqref{eq:OptProblem}--\eqref{eq:CostFunctionJ}:
\begin{enumerate}
\item Riemannian optimization methods (Section~\ref{sec:RiemanniansOptimization});
\item self-consistent field (SCF) algorithms exploiting the Aufbau principle proved in Theorem~\ref{thm:Aufbau} (Section~\ref{sec:SCF});
\item methods consisting in solving the convexified problem~\eqref{eq:convex_opt}, which, according to the results established in Theorem~\ref{thm:convex_pbm}, provides the global minimum of~\eqref{eq:OptProblem}--\eqref{eq:CostFunctionJ} if the matrix $\widetilde G_\star$ is gapped, and an hopefully good initial guess if $\widetilde G_\star$ is not gapped (Section~\ref{sec:convex_num}).
\end{enumerate}

\subsection{Riemannian optimization methods}
\label{sec:RiemanniansOptimization}

A straightforward approach to solve optimization problems on matrix manifolds such as the Grassmann manifold $\mathcal M={\rm Gr}(m,\R^M)$, is to use Riemannian optimization methods~\cite{absil2008optimization,boumal2023introduction}. The general idea of such techniques is to extend usual local optimization algorithms defined on linear spaces to manifolds using tools from differential geometry. Recall that the basic Riemannian gradient descent algorithm reads
\begin{align*}
    P_{k+1} = R_{P_k}\!\big(P_k -t_k \, {\rm grad}_{\mathcal{M}}J(P_k)\big),
\end{align*}
where $t_k > 0$ is the step length, ${\rm grad}_{\mathcal{M}}J(P_k)$ the Riemannian gradient at  $P_k$, and $R_{\bullet}$ a retraction. The more advanced Riemannian trust region methods can be formulated as 
$$
P_{k+1} = R_{P_k}\!\big(P_k +\eta_k \big),
$$
where $\eta_k$ is a solution to the intermediate optimization problem 
\begin{align*}
\min_{\eta \in T_{P_k}\mathcal M, \; \langle \eta,\eta\rangle_{P_k} \le \Delta_k} J(P_k) + \langle {\rm grad}_{\mathcal M} J(P_k), \eta\rangle_{P_k} + \frac 12 \langle \mathrm{Hess}_{\mathcal M} J(P_k)[\eta] ,\eta\rangle_{P_k},
\end{align*}
where $\Delta_k$ is an adaptively-chosen trust-region radius, $\mathrm{Hess}_{\mathcal M} J(P_k)$ the Riemannian Hessian at $P_k$, and $\langle \cdot,\cdot \rangle_\bullet$ the local metric (the Frobenius inner product in our case). 

\medskip

Although it is possible to use Riemannian optimization methods directly on Grassmann manifolds, it is sometimes more efficient from a numerical viewpoint to reformulate an optimization problem on a Grassmann manifold $\mathcal M:={\rm Gr}(m,\R^M)$ into an optimization problem on the Stiefel manifold 
$$
\mathcal N:={\rm St}(m,\R^M) := \{V \in \mathbb{R}^{M \times m} \; : \;  V^T V = I_{m}\}.
$$ 
Denoting by $J_{\rm St}: \mathbb{R}^{M \times m} \to \R$ the function defined by
$$
\forall V \in \R^{M \times m}, \quad J_{\rm St}(V) := J(VV^T),
$$
we have indeed
$$
\min_{P \in \mathcal M} J(P) = \min_{V \in \mathcal N} J_{\rm St}(V),
$$
and $P_\star$ is a minimizer of the LHS if and only if $P_\star=V_\star V_\star^T$, with $V_\star$ a minimizer of the RHS. Note, however, that if $V_\star$ is a minimizer of $J_{\rm St}$ on $\mathcal N$, then so is $V_\star U$ for all $U \in O(m)$. 

\medskip

To perform the simulations presented in Section~\ref{sec:numerics}, we used the Riemannian trust-region algorithm on the Grassmann manifold implemented in manopt.jl, a Riemannian optimization library in Julia~\cite{Bergmann2022}, with default parameters and convergence criterion $\|{\rm grad}_{\mathcal M}J(P_k)\| \le 10^{-12}$.

\subsection{Self-consistent field algorithms}
\label{sec:SCF}

In view of the Aufbau principle established in Theorem~\ref{thm:Aufbau}, the Roothaan algorithm~\cite{Roothaan1951}, originally introduced to solve the Hartree--Fock (HF) equations, provides a natural SCF algorithm to solve~\eqref{eq:OptProblem}--\eqref{eq:CostFunctionJ}. It reads
\begin{equation}
P_{k+1}^{\rm Rth} \in \mathop{\rm argmin}_{P \in \mathcal M} \tr(G(P_k^{\rm Rth})P) \quad (\mbox{recall that } G(P) = B - A P A). \label{eq:Roothaan}
\end{equation}
Each iteration is easy to perform if $L$ is not too large: we just have to compute $G(P_k^{\rm Rth})$, diagonalize this matrix in an orthonormal basis, and construct the orthogonal projector $P_{k+1}^{\rm Rth}$ on the vector space spanned by the lowest $m$ eigenvalues. The output of iteration $k$ is unique if and only if there is a gap between the $m^{\rm th}$ and $(m+1)^{\rm st}$ eigenvalue of $G(P_k^{\rm Rth})$. The Roothaan algorithm was investigated in detail for the (unrestricted) HF setting, which is also a quadratic Grassmann optimization problem for which the Aufbau principle is valid. It was shown~\cite{cances2000c,levitt2012b} that either the iterates converge, or they asymptotically oscillate between two states, none of them being a critical point of the HF problem. It was also shown numerically in~\cite{Cances_2021} that for non-quadratic cost functions, the Roothaan algorithm may either converge, or oscillate between two or more states, or have a chaotic behavior, depending on the cost function and the initial guess. In HF or Kohn--Sham (KS) Density Functional Theory (DFT) computations, the Roothaan algorithm is stabilized and accelerated using extrapolation (e.g. DIIS~\cite{Pulay1980,Pulay1982}), relaxation (e.g. ODA~\cite{cances2000b}), or a combination of both (e.g. EDIIS/DIIS scheme~\cite{Kudin2002}). Recall that the ODA is initialized by choosing $D_0 \in \mathcal K$ and reads
\begin{equation}
P_{k+1}^{\rm ODA} \in \mathop{\rm argmin}_{P \in \mathcal M} \tr(G(D_k)P), \qquad D_{k+1}=\mathop{\rm argmin}_{D \in [D_k,P_{k+1}^{\rm ODA}]} J(D), \label{eq:ODA}
\end{equation}
where $[D_k,P_{k+1}^{\rm ODA}]$ is the segment line in $\R^{M \times M}_{\rm sym}$ joining $D_k$ and $P_{k+1}^{\rm ODA}$. Parameterizing this set as
$$
[D_k,P_{k+1}^{\rm ODA}]= \{  (1-\alpha) D_{k}+\alpha P_{k+1}^{\rm ODA}, \; \alpha \in [0,1] \},
$$
and introducing the function $f_k:\R \to \R$ defined by 
$$
f_{k}(\alpha) = J((1-\alpha) D_{k}+\alpha P_{k+1}^{\rm ODA}),
$$
we see that, since $J$ is quadratic, $f_k$ is a quadratic polynomial. Its minimizer $\alpha_k$ on $[0,1]$ can therefore be computed analytically. Unless ODA has converged, $f'_k(0) \neq 0$, $\alpha_k$ is unique, and we have
$$
D_{k+1}^{\rm ODA}=(1-\alpha_k) D_{k}+\alpha_k P_{k+1}^{\rm ODA}.
$$
ODA can be viewed as a Frank--Wolfe algorithm for minimizing the HF or KS energy functional on the convex hull ${\rm CH}(\mathcal M)$. The reason why it actually provides suitable physical solutions to the original problem set on $\mathcal M$ is that (i) for (unrestricted) HF, all minimizers in ${\rm CH}(\mathcal M)$ are on $\mathcal M$ (this is a consequence of Lieb's variational principle for fermions~\cite{Lieb81}), (ii) for KS, it provides a solution of the extended Kohn--Sham model allowing fractional occupation numbers~\cite{Cances2001}. 
Remarkably, for the special case of the quadratic function $J$ defined in~\eqref{eq:CostFunctionJ}, we have the following result.

\begin{proposition} \label{prop:SCF} Consider a quadratic function $J$ of the form \eqref{eq:CostFunctionJ}. Then, 
\begin{enumerate}
\item the sequence $(P_k^{\rm ODA})_{k \in \N}$ generated by the ODA \eqref{eq:ODA} initialized with $D_0=P_0 \in \mathcal M$ coincides with the sequence $(P_k^{\rm Rth})_{k \in \N}$ generated by the Roothaan algorithm \eqref{eq:Roothaan} initialized with the same $P_0$ (when choosing the same output of iteration $k$ when the $m$-th and $(m+1)$-st eigenvalues of $G(P_k)$ are degenerate, and $\alpha_k=1$ when $f_k$ is constant);
\item denoting by $P_k:=P_k^{\rm Rth}=P_k^{\rm ODA}$, the sequence $(J(P_k))_{k \in \N}$ is non-increasing and converges to a critical value of $J$ on $\mathcal M$, and any accumulation point $P_\star$ of the sequence $(P_k)_{k \in \N}$ is a critical point of $J$ on $\mathcal M$ satisfying the Aufbau principle~\eqref{eq:Aufbau};
\item under the additional assumption that the matrix $A$ is invertible, the whole sequence $(P_k)_{k \in \N}$ converges to a critical point of $J$ on $\mathcal M$ satisfying the Aufbau principle~\eqref{eq:Aufbau}.
\end{enumerate}
\end{proposition}

\begin{proof} We adapt arguments from~\cite{cances2000b,cances2000c,levitt2012b}, where the Roothaan algorithm and the ODA were studied in the Hartree--Fock setting. First, we have
$$
f_{k}(\alpha) = a_{k} \alpha^2 + b_{k} \alpha + c_{k},
$$
with 
\begin{align*}
& c_k= f_k(0) =J(D_{k}), \\
& b_k= f_{k}'(0) = \langle \nabla J(D_k), P_{k+1}^{\rm ODA}-D_k \rangle =  \tr(G(D_k) (P_{k+1}^{\rm ODA}-D_k )) \leq 0, \\
& a_{k} = - \frac{1}{2} \tr( A (P_{k+1}^{\rm ODA}-D_k ) A(P_{k+1}^{\rm ODA}-D_k))  = - \frac{1}{2} \left\| A^{1/2} (P_{k+1}^{\rm ODA}-D_k ) A^{1/2} \right\|^2 \leq 0.
\end{align*}
which leads to the fact that the minimum of $f_k$ on the interval $[0,1]$ is reached at $\alpha_k = 1$. Hence, in this special case, ODA boils down to the Roothaan algorithm, and we thus have $D_k=P_k^{\rm ODA}=P_k^{\rm Rth}=:P_k$ for all $k \ge 1$. This proves the first assertion. 

\medskip

To prove the second one, we use the fact that as $\alpha_k=1$ for all $k$, we have that
\begin{align*}
J(P_{k+1}) &=   a_{k}  + b_{k} + c_{k} = J(P_{k}) + \tr(G(P_k) (P_{k+1}-P_k )) - \frac{1}{2} \left\| A^{1/2} (P_{k+1}-P_k ) A^{1/2} \right\|^2 \\
\end{align*}
so that 
\begin{align}
  0 \le \underbrace{- \tr(G(P_k) (P_{k+1}-P_k ))}_{\ge 0} +  \frac{1}{2} \left\| A^{1/2} (P_{k+1}-P_k ) A^{1/2} \right\|^2 \le J(P_{k})-J(P_{k+1}).
  \label{eq:ineq_preL}
\end{align}
This shows that the sequence $(J(P_k))_{k \in \N}$ is non-increasing, thus converges to some $J_\star \in \R$, and that
\begin{equation} \label{eq:bound_series_ODA}
\sum_{k=0}^{+\infty} \underbrace{(-\tr(G(P_k) (P_{k+1}-P_k )))}_{\ge  0}  < +\infty \quad \mbox{and} \quad \sum_{k=0}^{+\infty}  \left\| A^{1/2} (P_{k+1}-P_k ) A^{1/2} \right\|^2< \infty.
\end{equation}
Let $P_\star \in \mathcal M$ be an accumulation point of the sequence $(P_k)_{k \in \N}$. Since 
$$
P_{k+1} \in \mathop{\rm argmin}_{P \in \mathcal M} \tr(G(P_k)(P-P_k)),
$$
we must have
$$
\min_{P \in \mathcal M} \tr(G(P_\star) (P-P_\star )) =0;
$$
otherwise the sequence of nonnegative numbers $(-\tr(G(P_k) (P_{k+1}-P_k )))_{k \in \N}$ would not go to zero. This implies that $P_\star$ is a critical point of $J$ on $\mathcal M$ satisfying the Aufbau principle~\eqref{eq:Aufbau} and such that $J(P_\star)=J_\star$. The second assertion is proved.

Finally, if the matrix $A$ is invertible, we deduce from the second bound in~\eqref{eq:bound_series_ODA} and the equivalence of the norms in finite dimension that
$$
\sum_{k=0}^{+\infty}  \left\| P_{k+1}-P_k  \right\|^2< \infty.
$$
In particular, $\|P_{k+1}-P_k\| \to 0$. To prove that $(P_k)_{k \in \N}$ converges, we use the same arguments as in~\cite{levitt2012b}. Let $\Gamma:=\{P \in \mathcal M \; : \; J(P)=J_\star\}$ and $P \in \Gamma$. As $J$ is real-analytic, there exists,  for each point $P \in \Gamma$, $\delta_P > 0$, $\kappa_P > 0$ and $\theta_P \in (0,1/2]$ such that the \L ojasiewicz's inequality~\cite{Lojasiewicz} holds:
\begin{equation*}\label{eq:Lojasiewicz}
\forall P' \in \mathcal M \mbox{ s.t. } \|P'-P\|\le \delta_P, \quad  (J(P')-J_\star)^{1-\theta_P} \le \kappa_P \| \nabla_\mathcal MJ(P')\| = \kappa_P \|[G(P'),P']\|.
\end{equation*}
Since $\Gamma$ is non-empty (it contains $P_\star$) and compact, we obtain that there exist $\delta > 0$, $\kappa > 0$ and $\theta \in (0,1/2]$ such that 
$$ \forall P' \in \mathcal M \mbox{ s.t. } {\rm d}(P',\Gamma) \le \delta, \quad  (J(P')-J_\star)^{1-\theta} \le  \kappa \|[G(P'),P']\|.$$
By compactness $J(P_k)\to J_\star$ implies $d(P_k,\Gamma) \to 0$ and we thus have, using the fact that $[G(P_k),P_{k+1}]=0$, that there exists $k_0 \in \N$ such that
$$
\forall k \ge k_0, \quad (J(P_k)-J_\star)^{1-\theta} \le \kappa \|[G(P_k),P_k]\| = \kappa \|[G(P_k),P_k-P_{k+1}]\| \le \widetilde\kappa \|P_k-P_{k+1}\|,
$$
with $\widetilde\kappa=2\kappa \max_{P \in \mathcal M} \|G(P)\|_{\rm op}$.
We infer from the above inequality and the concavity of the function $\R_+ \ni x \mapsto x^\theta \in \R$ that for all $k \ge k_0$ for which $P_{k+1} \neq P_k$,
\begin{align*}
    (J(P_k)-J_\star)^\theta -  (J(P_{k+1})-J_\star)^\theta & \ge \theta (J(P_k)-J_\star)^{\theta-1} \left( J(P_k)-J(P_{k+1}) \right) \\
    & \ge \frac{\theta}{\widetilde\kappa} \frac{\left( J(P_k)-J(P_{k+1}) \right)}{\|P_k-P_{k+1}\|}.
\end{align*}
Using \eqref{eq:ineq_preL} and the invertibility of $A$, we obtain that there exists $\alpha > 0$ (one can take $\alpha=\frac 12 \lambda_{\rm min}(A)^2$) such that
$$
J(P_k)-J(P_{k+1}) \ge \alpha \|P_k-P_{k+1}\|^2,
$$
and this inequality obviously also holds for $P_{k+1}=P_k$.
We thus get 
$$
\forall k \ge k_0, \quad (J(P_k)-J_\star)^\theta -  (J(P_{k+1})-J_\star)^\theta \ge \frac{\alpha\theta}{\widetilde\kappa} 
\|P_k-P_{k+1}\|.
$$
As the left-hand side is a summable series, so is the right-hand side. This implies that the sequence $(P_k)_{k \in \N}$ converges.
\end{proof}

\subsection{Frank--Wolfe algorithm for the convexified problem}
\label{sec:convex_num}

To solve the convexified problem~\eqref{eq:convex_opt}, we use the Frank--Wolfe algorithm, which reads in this specific setting
\begin{equation}
P_{k+1}^{\rm conv} \in \mathop{\rm argmin}_{P \in \mathcal M} \tr(\widetilde G(D_k^{\rm conv})P), \qquad D_{k+1}^{\rm conv}=\mathop{\rm argmin}_{D \in [D_k^{\rm conv},P_{k+1}^{\rm conv}]} \widetilde J(D), \label{eq:ODA_conv}
\end{equation}
with 
$$
\widetilde G(D) = \nabla \widetilde J(D) = C- \frac 12[[A,D],A].
$$
We have that
$$
D_{k+1}^{\rm conv}=(1-\beta_k) D_{k}^{\rm conv}+\beta_k P_{k+1}^{\rm conv},
$$
where the optimal damping parameter $\beta_k$ is obtained by solving
$$
\beta_k \in \mathop{\rm argmin}_{\beta \in [0,1]} \widetilde J((1-\beta) D_{k}^{\rm conv}+\beta P_{k+1}^{\rm conv})= \mathop{\rm argmin}_{\alpha \in [0,1]} \widetilde f_k(\beta),
$$
with 
$$
\widetilde f_k(\beta) = \widetilde a_k \beta^2 + \widetilde b_k \beta + \widetilde c_k,
$$
and
$$
\widetilde a_k= \frac 14 \| [A,P_{k+1}^{\rm conv}-D_{k}^{\rm conv}] \|^2 , \quad \widetilde b_k=\tr(\widetilde G(D_k^{\rm conv})(P_{k+1}^{\rm conv}-D_{k}^{\rm conv})), \quad \widetilde c_k= \widetilde J(D_{k}^{\rm conv}).
$$
The solution is given by $\beta_k=\min(1,-\widetilde b_k/(2\widetilde a_k))$ if $\widetilde a_k > 0$ and $\beta_k=1$ if $\widetilde a_k=0$.

\section{Numerical illustrations and applications}
\label{sec:numerics}

The numerical experiments below serve two complementary purposes. Section~\ref{sec:random} uses low-dimensional model problems to illustrate phenomena that may occur for the quadratic Grassmann optimization problem, including the existence of non-global local minima, convergence of SCF methods to such minima, and limitations of the convexified formulation. Section~\ref{sec:molecules} then examines whether and how these phenomena arise in DMET bath-construction problems for chemically and physically motivated systems, namely the benzene molecule and the Hubbard model.

\subsection{Analysis of low-dimensional illustrative examples}
\label{sec:random}

The purpose of this section is to show that 
\begin{itemize}
\item the functional $J$ may have local, non-global minima on $\mathcal M$;
\item these local minima have attraction basins with positive measure not only for Riemannian optimization methods, but also for the SCF algorithm~\eqref{eq:Roothaan};
\item while the Roothaan algorithm applied to the minimization of $J$ on $\mathcal M$ always converges when $A$ is invertible (in a sense specified in Proposition~\ref{prop:SCF}), this is not the case when this algorithm is applied to the minimization of $\widetilde J$ on $\mathcal M$, although the functions $J$ and $\widetilde J$ coincide on $\mathcal M$;
\item the set of minimizers of $\widetilde J$ on $\mathcal K$ may contain no point belonging to $\mathcal M$, in which case minimizing $\widetilde J$ on $\mathcal K$ using e.g. the ODA~\eqref{eq:ODA_conv}, does not provide the global minimizer on $J$ on~$\mathcal M$.
\end{itemize}

\medskip

\noindent
Consider first the case when $M=2$ and $m=1$. The case when $A=\lambda I_2$ with $\lambda \ge 0$ is trivial. Without loss of generality, we can assume that $A$ is diagonal with $A_{11}=2$, $A_{22}=1$ and $C$ is traceless, and consider the families of matrices 
\begin{align*}
    A = \begin{pmatrix}
        2 & 0 \\ 0 & 1
    \end{pmatrix}, \quad B_{b,c} = \begin{pmatrix}
        c+2 & b \\ b & -c+\frac{1}2 
    \end{pmatrix}, \quad C_{b,c} = \begin{pmatrix}
        c & b \\ b & -c
    \end{pmatrix}, \quad b,c \in \R,
\end{align*}
and the families of functions 
$$
J_{b,c}(D):=\tr(B_{b,c} D) - \frac 12 \tr(ADAD) \quad \mbox{and} \quad \widetilde J_{b,c}(D):=\tr(C_{b,c} D) + \frac 14 \|[A,D]\|^2,  \quad b,c \in \R
$$
whose landscapes strongly depend on the values for $b,c$, see Fig.~\ref{fig:jbc_lanscapes}.

The set $\mathcal M$ is diffeomorphic to the unit circle (the set of rank-$1$ orthogonal projectors in $\R^2$ can be represented as a great circle of the Bloch sphere): any $P \in \mathcal M$ can be written in a unique way as 
\begin{equation}
P_{x,z} = \frac 12 \left( I_2 + x \sigma_x + z \sigma_z \right) \quad \mbox{with} \quad x,z \in \R, \; x^2+z^2=1, \quad \sigma_x=\begin{pmatrix} 0 & 1 \\ 1 & 0 \end{pmatrix}, \quad 
\sigma_z=\begin{pmatrix} 1 & 0 \\ 0 & -1 \end{pmatrix}. \label{eq:param_xz}
\end{equation}
Likewise, the set ${\rm CH}(\mathcal M)$ is diffeomorphic to the closed unit disc of $\R^2$: any $D \in {\rm CH}(\mathcal M)$ can be written in a unique way as 
$$
D_{x,z} = \frac 12 \left( I_2 + x \sigma_x + z \sigma_z \right) \quad \mbox{with} \qquad x,z \in \R, \; x^2+z^2 \le 1.
$$
When the variable $(x,z)$ span $\R^2$, the matrix $D_{x,z}$ defined by the above formula spans the affine space of trace-$1$ $2 \times 2$ real-symmetric matrices. The restriction of $J$ and $\widetilde J$ to this affine space can be written in terms of the $(x,z)$ variable. We have for all $(x,z) \in \R^2$, $J_{b,c}(D_{x,z})=j_{b,c}(x,z)$, $\widetilde J_{b,c}(D_{x,z})=\widetilde j_{b,c}(x,z)$ with $j_{b,c},\widetilde j_{b,c}: \R^2 \to \R$ given by 
\begin{align}
& j_{b,c}(x,z) =  \frac{5}8 + b x + c z - \frac{1}2  x^2  - \frac {5}8 z^2, \label{eq:jbc}
\\
    &\widetilde j_{b,c}(x,z) = b x + c z + \frac 1 8  x^2. \label{eq:tjbc}
\end{align}
It is clear that on $\mathbb S^1:=\{(x,z) \in \R^2 \, : \, x^2+z^2=1\}$, we have $\widetilde j_{b,c}(x,z)=j_{b,c}(x,z)$ as expected.
The following angular parameterization of $\mathcal M$ will also be useful:
$$
P_\theta := \frac 12  \left( I_2 + \sin\theta \, \sigma_x + \cos \theta \,  \sigma_z \right) = \begin{pmatrix} \frac{1+\cos(\theta)}2 & \frac{\sin(\theta)}2 \\  \frac{\sin(\theta)}2 & \frac{1-\cos(\theta)}2 \end{pmatrix} = \begin{pmatrix} \cos^2\left(\frac\theta 2\right) & \cos\left(\frac\theta 2\right)\sin\left(\frac\theta 2\right) \\  \cos\left(\frac\theta 2\right)\sin\left(\frac\theta 2\right) & \sin^2\left(\frac\theta 2\right) \end{pmatrix},
$$
with $\theta \in \R/2\pi\Z$. We have in particular 
$$
P_0:=\left( \begin{array}{cc} 1 & 0 \\  0 & 0  \end{array} \right), \quad P_{\pi}:=\left( \begin{array}{cc} 0 & 0 \\ 0 & 1 \end{array} \right), \quad P_{\frac\pi 2}:= \frac 12 \left( \begin{array}{cc} 1 & 1 \\ 1 & 1  \end{array} \right), \quad P_{\frac{3\pi} 2}:= \frac 12 \left( \begin{array}{rr} 1 & -1 \\ -1 & 1  \end{array} \right).
$$
We set
\begin{align}
& j_{b,c}(\theta) :=J_{b,c}(P_\theta)=\widetilde J_{b,c}(P_\theta)=b \sin\theta + c \cos\theta + \frac 1 8 \sin^2\theta. \label{eq:function_fbc}
\end{align}
\begin{figure}[ht]
    \centering   \includegraphics[width=0.65\linewidth]{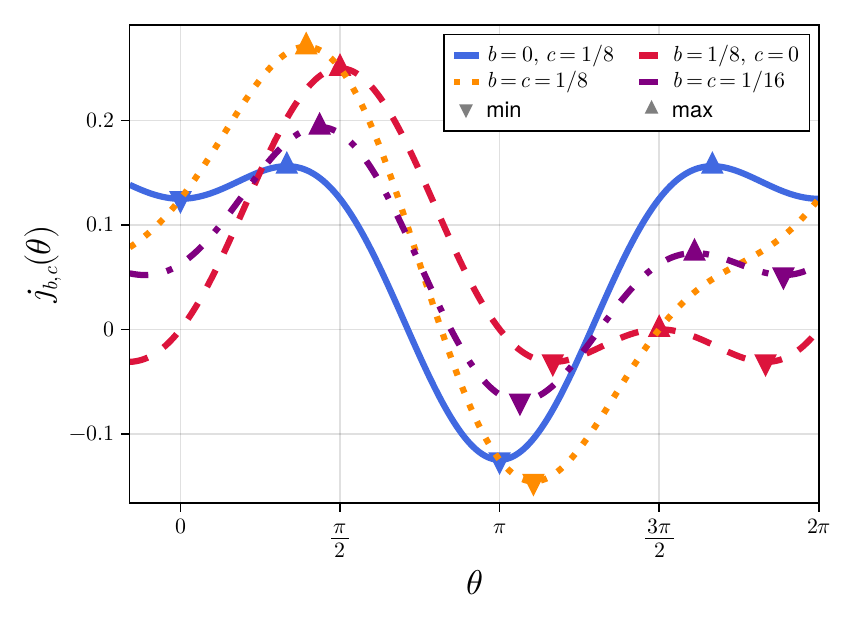}
    \caption{Examples of landscapes of the cost function $j_{b,c}$ in \eqref{eq:function_fbc} for different values of the parameters $(b,c)$. We highlight the local minima (resp. maxima) of $j_{b,c}$ using the symbol $\blacktriangledown$ (resp.~$\blacktriangle$).}
\label{fig:jbc_lanscapes}
\end{figure}
\medskip

\noindent
{\bf Critical points of $J_{b,c}$ on $\mathcal M$}

\medskip

\noindent
Using the parameterization~\eqref{eq:param_xz} and the expression~\eqref{eq:tjbc}, we see that $P_{x,z}$ is a critical of $J_{b,c}$ on $\mathcal M$ if and only if there exists $\lambda \in \R$ such that
\begin{align*}
    b+\frac x4=\lambda x, \quad c=\lambda z, \quad x^2+z^2=1.
\end{align*}
It follows from elementary geometrical arguments that $J_{b,c}$ has at most 4 critical points on $\mathcal M$. In particular,
\begin{itemize}
\item If $b=c=0$, $P_0$ and $P_\pi$ are global minimizers ($J_{0,0}(P_0)=J_{0,0}(P_\pi)=0$), while $P_{\pi/2}$ and $P_{3\pi/2}$ are global maximizers ($J_{0,0}(P_{\pi/2})=J_{0,0}(P_{3\pi/2})=\frac 18$).
    \item If $c=0$ and $b \neq 0$, then 
$P_{\frac\pi 2}$ and $ P_{\frac{3\pi} 2}$
are always critical points, and it holds 
$$
J_{b,0}(P_{\pi/2}) = b+\frac 18 \quad \mbox{and} \quad  J_{b,0}(P_{3\pi/2}) = -b+\frac 18.
$$
If $|b| \ge 1/4$, $J_{b,0}$ has no other critical point. If $|b| < 1/4$, $J_{b,0}$ has two other critical points: $P_{-\arcsin(4b)}$ and $P_{\pi+\arcsin(4b)}$, which are in fact the two global minimizers 
($J(P_{-\arcsin(4b)})= J(P_{\pi+\arcsin(4b)})=-2b^2$), while $P_{\pi/2}$ and $P_{3\pi/2}$ are local maximizers (one of them is global, the other is not).
\item If $b=0$ and $c \neq 0$, $P_{0}$ and $ P_{\pi}$
are always critical points, and it holds 
$$
J_{0,c}(P_{0}) = c\quad \mbox{and} \quad  J_{0,c}(P_{\pi}) = -c.
$$
If $|c| \ge 1/4$, $J_{0,c}$ has no other critical point. If $|c| < 1/4$, $J_{0,c}$ has two other critical points: $P_{\arccos(4c)}$ and $P_{2\pi -\arccos(4c)}$, which are in fact the two global maximizers 
($J(P_{\arccos(4c)})= J(P_{2\pi-\arccos(4c)})=\frac 18+2c^2$), while $P_{0}$ and $P_{\pi}$ are local minimizers (one of them is global, the other is not).
\end{itemize}

\medskip

\noindent
{\bf Existence of local, non-global minima of $J$ on $\mathcal M$}

\medskip

\noindent
It follows from the above discussion that for $b=0$ and $0 < c < 1/4$,  $P_{0}$ is a local, non-global minimizer ($J(P_{0})=c$), while  $P_{\pi}$ is the global minimizer ($J(P_{\pi})=-c$).

\bigskip

\noindent
{\bf The Roothaan algorithm~\eqref{eq:Roothaan} may converge to local non-global minima}

\medskip

\noindent
Proposition~\ref{prop:SCF} shows that the Roothaan algorithm always converge. However, it does not necessarily converge to a global minimum. To prove that, let us observe that, in the simple setting considered here, the Roothaan algorithm can be reformulated as a dynamical system on $\mathbb S^1$. 
Using the angular parameterization, it reads as follows: $\theta_k$ being given, construct 
$$
G(P_{\theta_k}) = \begin{pmatrix} c-2\cos\theta_k & b - \sin\theta_k \\ b-\sin\theta_k & -c+\frac 12 \cos \theta_k \end{pmatrix} = -\frac 34 \cos \theta_k \, I_2 + (b-\sin\theta_k) \sigma_x + \left(c-\frac 54\cos\theta_k\right) \sigma_z, 
$$
and define $\theta_{k+1}$ as the angle of the orthogonal projector on the eigenspace of $G(P_{\theta_k})$ associated with its lowest eigenvalue. Unless $\cos\theta_k=\frac 45 c$ and $\sin\theta_k=b$ (which is only possible if $b^2+\frac{16}{25}c^2=1$), the Roothaan iterate is always uniquely defined and such that
\begin{align*}
&\cos\theta_{k+1}=\frac{\frac 54\cos\theta_k-c}{\left( \left(\frac 54\cos\theta_k-c\right)^2+ \left( \sin\theta_k-b\right)^2\right)^{1/2}}, \quad \sin\theta_{k+1}=\frac{\sin\theta_k-b}{\left( \left(\frac 54\cos\theta_k-c\right)^2+ \left( \sin\theta_k-b\right)^2\right)^{1/2}}.
\end{align*}
For $b=0$ and $0 < c < \frac 14$, $\theta=0$ is a local, non-global, minimizer of $f_{0,c}$, and a fixed point of the Roothaan algorithm. In addition, in the vicinity of $\theta=0$, the Roothaan algorithm for $J_{0,c}$ reads
$$
\theta_{k+1}= F_c(\theta_k) \qquad \mbox{with} \qquad F_c(\theta):= \arcsin\left( \frac{\sin\theta}{\left( \left(\frac 54\cos\theta-c\right)^2+  \sin^2\theta \right)^{1/2}}\right).
$$
As $0 < F'_c(0)=\frac{1}{\frac 54-c} < 1$ for all $0 < c < 1/4$, the function $F_c$ is contractive in the neighborhood of $\theta=0$, which proves that the basin of attraction of $\theta=0$ for the Roothaan algorithm contains an open neighborhood of $0$. Iterations of the Roothaan algorithm applied to $j_{b,c}$ are shown in Fig.~\ref{fig:Roothaan_b_diff_c} (center column).

\medskip

\begin{figure}[ht]
    \centering
    \begin{minipage}{0.345\textwidth}
        \includegraphics[width=1.0\textwidth]{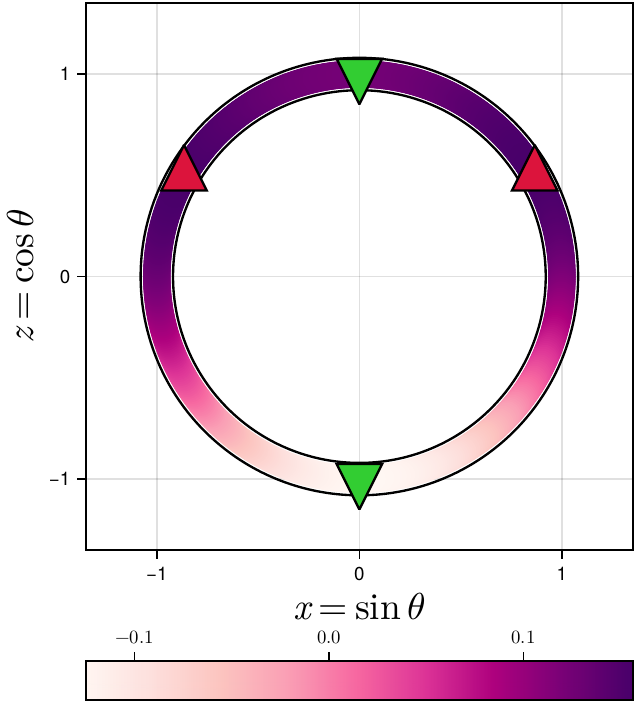}
    \end{minipage}
    \begin{minipage}{0.31\textwidth}
        \includegraphics[width=1.0\textwidth]{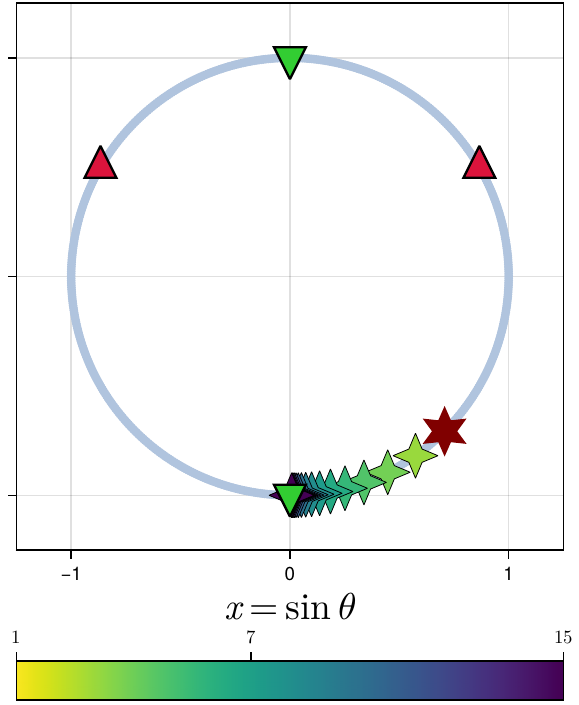}
    \end{minipage}
        \begin{minipage}{0.31\textwidth}
        \includegraphics[width=1.0\textwidth]{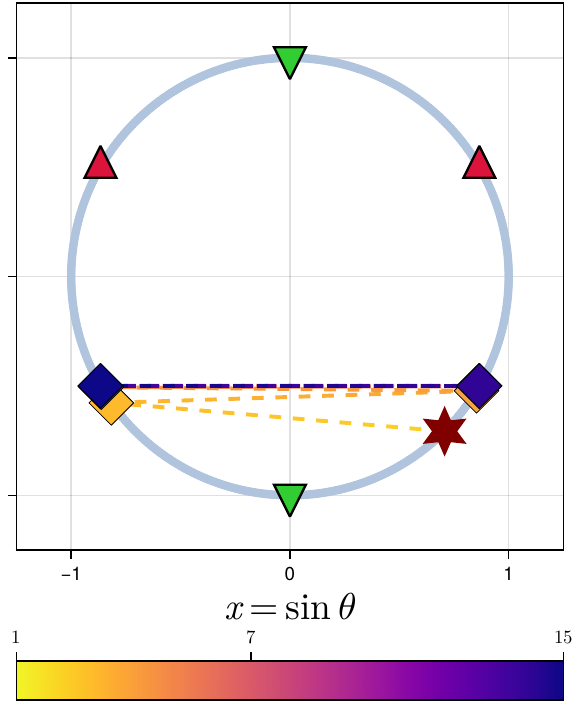}
    \end{minipage}  
    
    \begin{minipage}{0.345\textwidth}
        \includegraphics[width=1.0\textwidth]{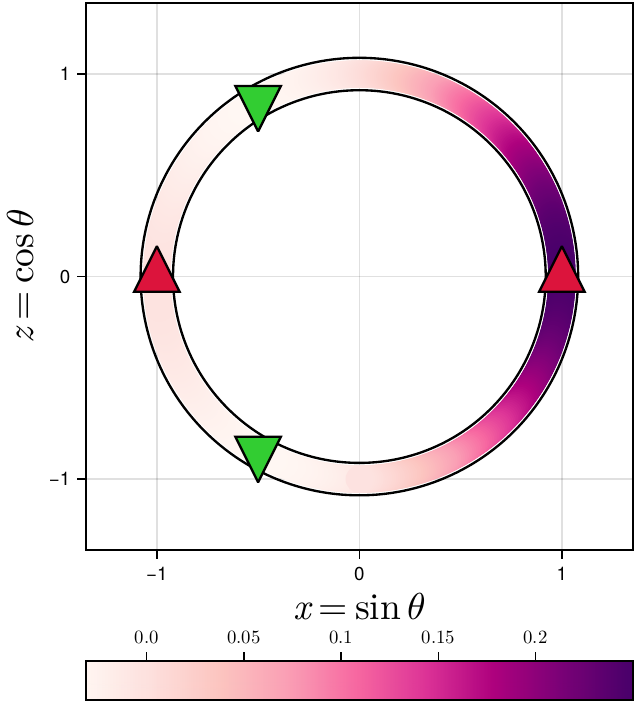}
    \end{minipage}
    \begin{minipage}{0.315\textwidth}
        \includegraphics[width=1.0\textwidth]{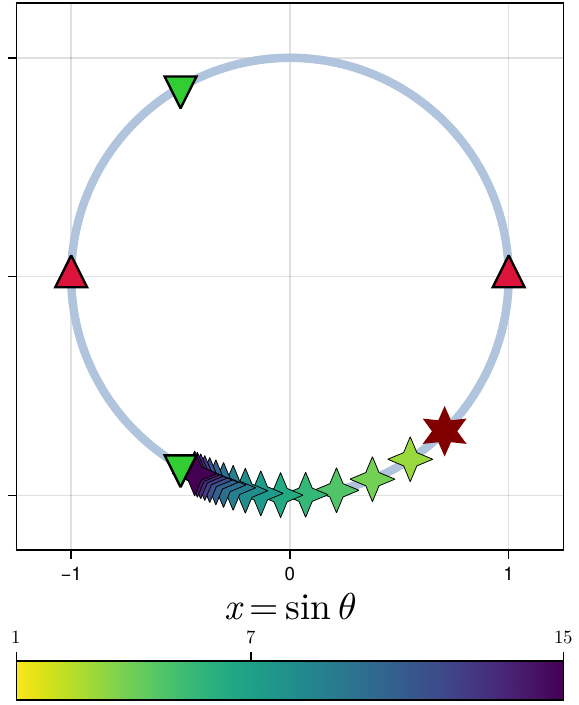}
    \end{minipage}   
        \begin{minipage}{0.315\textwidth}
        \includegraphics[width=1.0\textwidth]{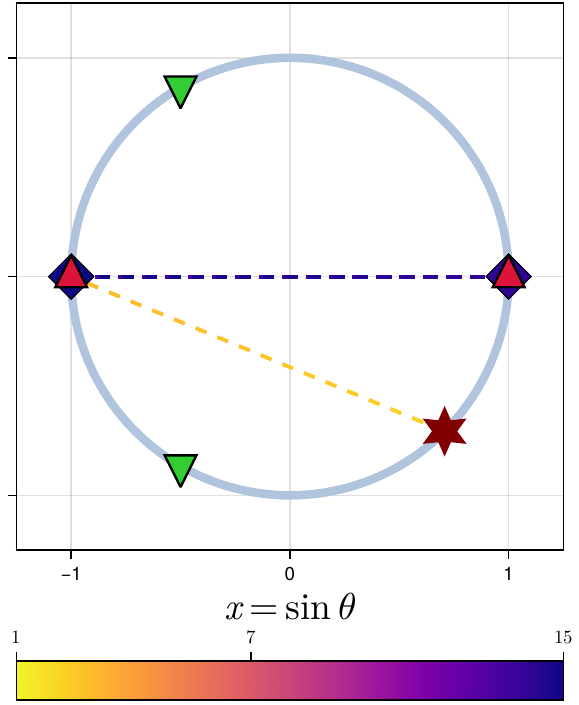}
    \end{minipage}  
    
    \caption{Colormap of the restriction of $j_{b,c}$ (or equivalently $\widetilde j_{b,c}$) to the unit circle representing the Grassmann manifold for $M=2$, $m=1$ (left column), and first 15 iterates of the Roothaan algorithm applied to $j_{b,c}$ (center column, \FourStar), and to $\widetilde j_{b,c}$ (right column, $\blacklozenge$), with initial guess ({\textcolor{red}{\small{\SixStar}}}) corresponding to $\theta_0=3\pi/4$ for different values of the parameters $(b,c)$: $(b,c)=(0,1/8)$ (top), $(b,c)=(1/8,0)$ (bottom). The symbol $\blacktriangledown$ (resp. $\blacktriangle$) indicates a local minimum (resp. maximum). \label{fig:Roothaan_b_diff_c}}
\end{figure}

\begin{figure}[ht]
    \centering
 
    \begin{minipage}{0.345\textwidth}
        \includegraphics[width=1.0\textwidth]{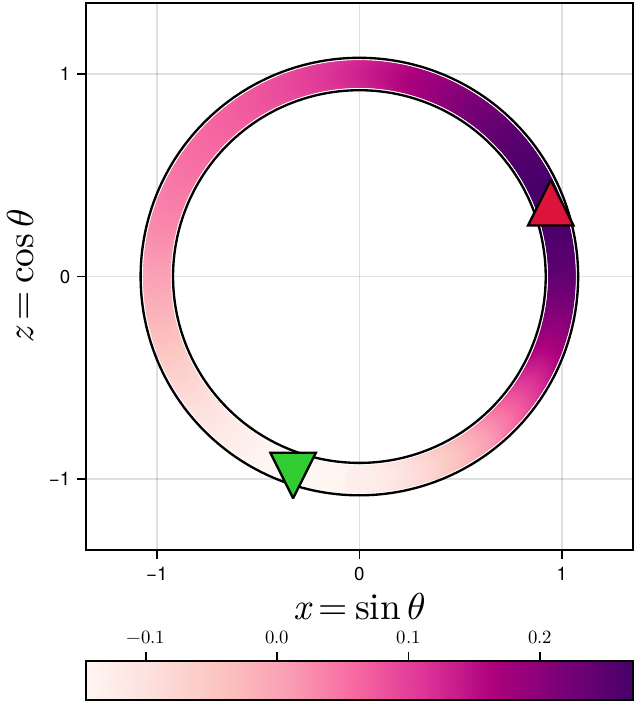}
    \end{minipage}
    \begin{minipage}{0.31\textwidth}
        \includegraphics[width=1.0\textwidth]{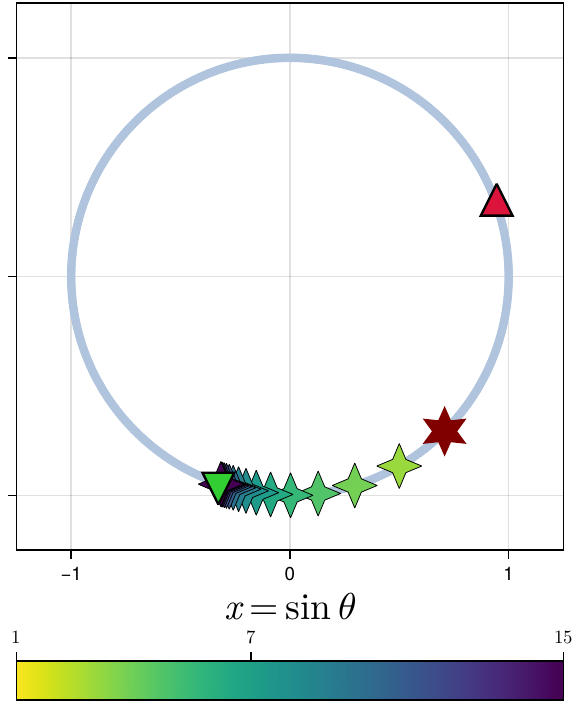}
    \end{minipage}   
        \begin{minipage}{0.31\textwidth}
        \includegraphics[width=1.0\textwidth]{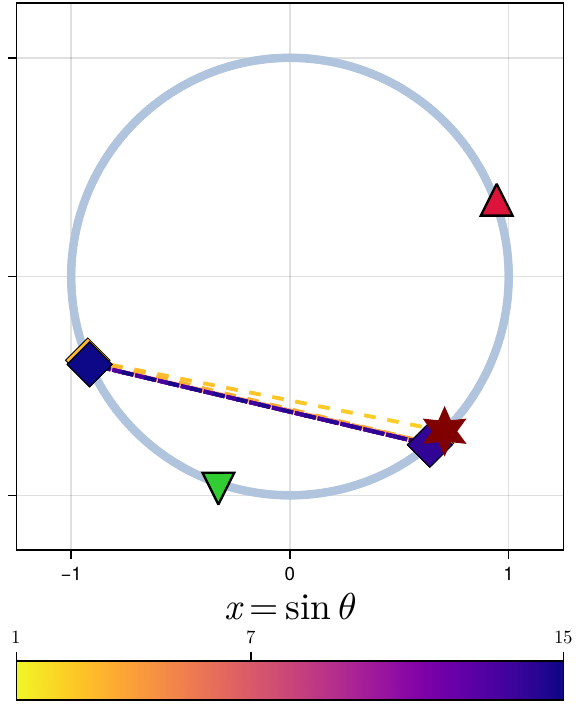}
    \end{minipage}
    
    \begin{minipage}{0.345\textwidth}
        \includegraphics[width=1.0\textwidth]{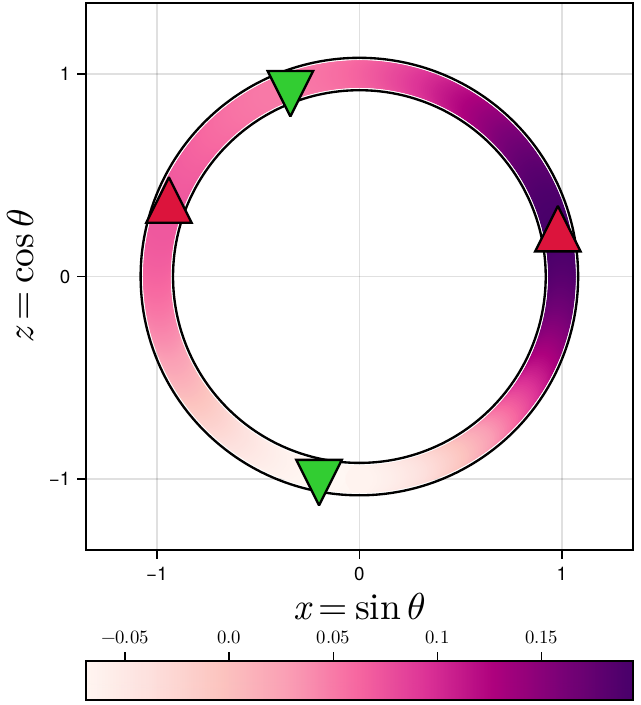}
    \end{minipage}
    \begin{minipage}{0.31\textwidth}
        \includegraphics[width=1.0\textwidth]{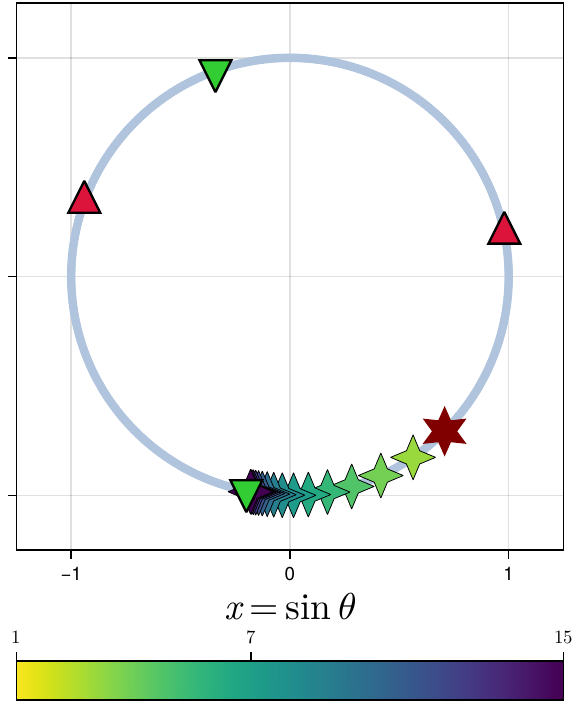}
    \end{minipage}
        \begin{minipage}{0.31\textwidth}
        \includegraphics[width=1.0\textwidth]{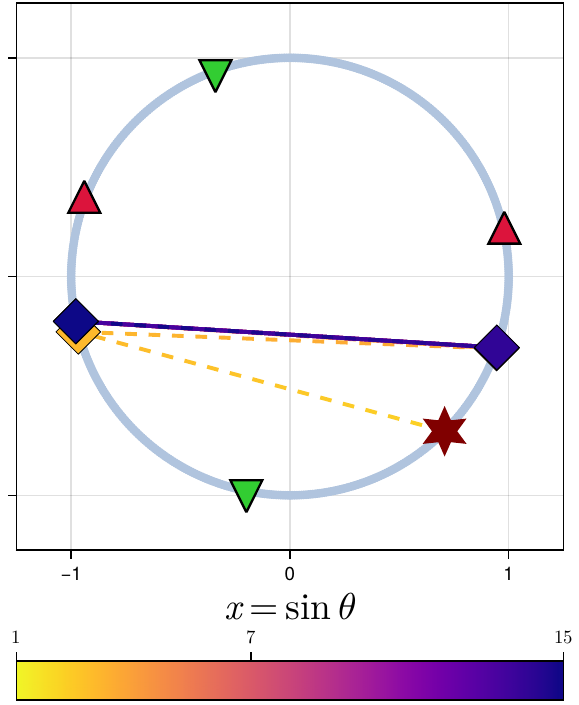}
    \end{minipage}   
    \caption{Same setup as in Fig.~\ref{fig:Roothaan_b_diff_c}, but for parameters $(b,c)=(1/8,1/8)$ (top) and $(b,c)=(1/16,1/16)$ (bottom). \label{fig:Roothaan_b=c}}
\end{figure}

\noindent
{\bf The Roothaan algorithm applied to $\widetilde J$ does not always converge}

\medskip

\noindent
For $c=0$ and $0 < |4 b| < 1$, we have 
$$
\widetilde j_{b,0}(x,z)=  b x + \frac 18 x^2 = \frac 18 (x+4 b)^2 - 2  b^2.
$$

This function therefore has two global minimizers on $\mathbb S^1$, given by 
$$
X_{b,+}:=\left( \begin{array}{c} -4 b \\ \sqrt{1-16 b^2} \end{array} \right), \quad X_{b,-}:=\left( \begin{array}{c}  -4 b \\ -\sqrt{1-16 b^2} \end{array} \right)
$$
corresponding respectively to the global minimizers $P_{-\arcsin(4b)}$ and $P_{\pi+\arcsin(4b)}$ of $J_{b,0}$ on $\mathcal M$. The Roothaan algorithm can be formulated as an alternate direction method~\cite{cances2000c}. In terms of the function $\widetilde j_{b,0}$, it reads
$$
x_{k+1} \in  \mathop{\rm argmin}_{x \in [-1,1]} \left( \frac{ b x}2 + \frac 18 x_kx \right) = \mathop{\rm argmin}_{x \in [-1,1]} \left( (4  b+x_k) x \right) = \left| \begin{array}{ll} \{-1\} & \mbox{ if } 4 b + x_k > 0, \\\ \{1\} & \mbox{ if } 4 b + x_k < 0, \\ {[}-1,1{]} & \mbox{ if } 4 b + x_k = 0. \end{array} \right. 
$$
It follows that for almost all initial guesses, the Roothaan algorithm applied $\widetilde J_{b,0}$ will oscillate between the states $P_{\frac\pi 2}$ and $P_{\frac{3\pi} 2}$. It is clearly visible in Fig.~\ref{fig:Roothaan_b=c} (right column) that the Roothaan algorithm applied on $\tilde j_{b,c}$ oscillates between two points, which, in one the considered cases, appear to be critical points.

\medskip

\noindent
{\bf Minimizing $\widetilde J$ over $\mathcal K$ does not always provide the global minimizers of $J$ on $\mathcal M$}

\medskip

\noindent
We now consider the case when $M = 3$ and $m =1$ and introduce the matrices
\begin{align*}
    A = \begin{pmatrix}
        a_1 & 0 &0\\0 & a_2 & 0\\ 0&0&a_3
    \end{pmatrix}, \quad C= \begin{pmatrix}
        \beta & - \frac \alpha2 (a_2 - a_1)^2 & 0 \\ -\frac \alpha2 (a_2 - a_1)^2 & \beta & -\frac \alpha2 (a_3 - a_2)^2\\ 0  & -\frac \alpha2 (a_3 - a_2)^2 & \beta
    \end{pmatrix},
\end{align*}
with $0  \le a_1 < a_2 < a_3$, $\alpha > 0$, $\beta \in \R$. We have for all $D \in \R^{3 \times 3}_{\rm sym}$ such that $\tr(D)=1$,
\begin{align*}
    \widetilde J(D) &= \tr(CD) + \frac 14 \lVert [A,D] \rVert^2 = \sum_{i,j =1}^{M} C_{ij} D_{ij} + \frac 14 \sum_{i,j=1}^{M} (a_i - a_j)^2 D_{ij}^2 \\  &=\sum_{i = 1}^{M}C_{ii}D_{ii}  + \frac 14 \sum_{i \neq j}(a_i - a_j)^2 \left( D_{ij} + 2 \frac{C_{ij}}{(a_i - a_j)^2} \right)^2 - \sum_{i\neq j }
\frac{C_{ij}^2}{(a_i - a_j)^2} \\
&= \beta + \frac 12 (a_2 - a_1)^2 (D_{12}-\alpha)^2 + \frac 12 (a_3 - a_2)^2 (D_{23}-\alpha)^2 + \frac 12(a_3-a_1)^2 D_{13}^2 \\ & \quad - \frac 12 \alpha^2 ((a_2 - a_1)^2+(a_3 - a_2)^2).
\end{align*}
We thus have 
$$
\mathcal D:= \mathop{\rm argmin}_{\substack{ D \in \R^{3 \times 3}_{\rm sym} \; \\ \tr(D)=1}} \widetilde J(D) = \left\{     D = \begin{pmatrix}
        n_1 & \alpha& 0 \\ \alpha&n_2&\alpha\\0&\alpha &n_3
    \end{pmatrix}, \; n_1,n_2,n_3 \in \R, \; n_1+n_2+n_3=1 \right\}.
$$
Since for any $D \in \mathcal D$, $[D^2]_{13}=\alpha^2 > 0 = D_{13}$, we have $\mathcal D \cap \mathcal M = \emptyset$, while for $\alpha > 0$ small enough, we have
$$
 \begin{pmatrix}
        1/3 & \alpha& 0 \\ \alpha& 1/3 &\alpha\\0&\alpha &1/3
    \end{pmatrix} \in \mathcal D \cap \mathcal K.
$$
It follows that for $\alpha > 0$ small enough, the set of minimizers of $\widetilde J$ on $\mathcal K$ does not contain any point of $\mathcal P$. The global minimizer of $J$ on $\mathcal M$ can therefore not be found by minimizing the convex functional $\widetilde J$ over the convex set $\mathcal K$. Note that for this example, the matrix $\widetilde G_\star$ introduced in Theorem~\ref{thm:convex_pbm} is given by $\widetilde G_\star=\beta I_3$.

\medskip

\noindent
Taking now $a_1=1$, $a_2=2$, $a_3=3$, $\beta=0$, $\alpha=1/2$, we obtain that in this case $\mathcal D \cap \mathcal K=\emptyset$. The matrix $\widetilde G_\star$ and the set of minimizers of $\widetilde J$ on $\mathcal K$ are given by
$$
\widetilde G_\star =  \begin{pmatrix}
        0 & -1/9 & 1/9 \\ -1/9 & 0 & -1/9 \\ 1/9 & - 1/9 & 0
    \end{pmatrix} \quad  \mbox{and} \quad \mathop{\rm argmin}_{D \in \mathcal K} \widetilde J(D) = \left\{  \begin{pmatrix} 2/9 &  5/18 &  1/18 \\
 5/18 &   5/9 &   5/18 \\
 1/18 &  5/18 &  2/9 \end{pmatrix} \right\}.
$$
The eigenvalues of $\widetilde G_\star$ are $\mu_1=\mu_2 = -1/9$, $\mu_3=2/9$. In agreement with the theoretical results, there is no gap between the first and second eigenvalues of $\widetilde G_\star$ (recall that, here, $m=1$).

\subsection{DMET bath construction for chemical and physical examples}
\label{sec:molecules}

We now present and analyze the results of the various numerical methods we introduced when applied to the solution of our target problem in the context of the DMET method. Indeed, as explained in section \ref{sec:DMET}, our cost function typically arises when constructing the bath of the DMET method.

We pick two relevant examples: a chemical system, the benzene molecule, and a physical system, the Hubbard chain. 
For both examples, the matrices $A$ and $B$ that define cost function are obtained using formulae \eqref{rdm1_basis1}--\eqref{eq:BCP-C}, where $\gamma$ is the exact ground-state 1-RDM of the system, or rather a good approximation of it. The motivation for this choice of $\gamma$ is that in the final steps of the DMET iteration loop, the trial 1-RDM is expected to be a good approximation of the ground-state 1-RDM.

\begin{figure}[H]
\centering
\begin{minipage}{0.45\textwidth}
\centering
\includegraphics[width=0.60\linewidth]{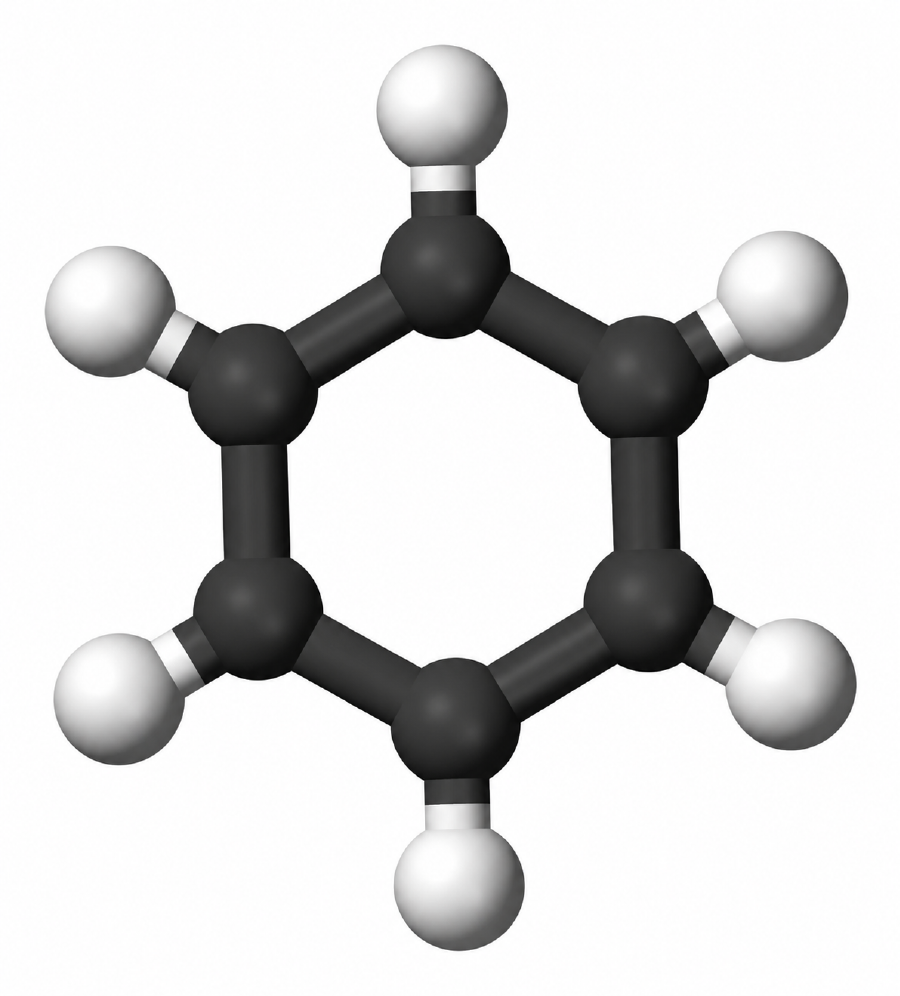}

\end{minipage}
\begin{minipage}{0.40\textwidth}
\centering
\begin{tikzpicture}[scale=0.5,>=stealth]

\def\n{4}
\def\d{3}  
\definecolor{sitecol}{RGB}{100,180,100}
\definecolor{tcol}{RGB}{200,60,100}
\definecolor{Ucol}{RGB}{50,90,200}

\foreach \i in {1,...,\n}
{
  \node[circle, draw, fill=sitecol, minimum size=1.2pt] (s\i)
    at ({(\i-1)*\d},0) {};

  \draw[Ucol, line width=1.1pt] (s\i) circle (0.48);
}

\foreach \i in {1,...,\numexpr\n-1\relax}
{
  \pgfmathtruncatemacro{\next}{\i+1}

  \draw[<->, very thick, tcol] (s\i) -- (s\next);

  \path (s\i) -- (s\next) coordinate[midway] (m);
  \node[tcol, above] at (m) {\Large $t$};
}

\foreach \i in {1,...,\n}
{
  \node[Ucol] at ({(\i-1)*\d},1.4) {\Large $U$};
}

\end{tikzpicture}
\end{minipage}
\caption{Physical systems used as examples of DMET bath-construction problems. \textbf{Left}: Benzene molecule, C$_6$H$_6$ (carbon atoms are in black, hydrogen atoms in white).
\textbf{Right}: The Hubbard chain with hopping $t$ and interaction $U$, here for $4$ sites and hence $L=8$ spin-orbitals. \label{fig:physical_systems}}
\end{figure}

\subsubsection{The benzene molecule}
Numerical tests on molecular systems can be done using the following methodology: 
\begin{description}
    \item[Step 1:] The ground-state spin-unpolarized 1-RDM of the molecular system under consideration is computed at the chosen level of theory (e.g. Hartree--Fock, configuration interaction single-double - CISD -, coupled-cluster single-double - CCSD -, complete active space self-consistent field - CASSCF -) in the chosen gaussian atomic basis set, using the PySCF Python library~\cite{sun2020recent,sun2018pyscf,sun2015libcint}. Using L\"owdin orthonormalization, the non-orthogonal, atom-centered, gaussian-orbital basis is transformed into an orthonormal, and still reasonably localized, basis, which we call the L\"owdin basis for convenience. In particular, each basis function of the L\"owdin basis is essentially localized around a specific atom, which allows us to obtain a decomposition~\eqref{eq:fragmentation} of the variational approximation space by simply partitioning the atoms in the molecule. We denote by $\gamma_{\rm GS}$ the spin-unpolarized ground-state 1-RDMs of the molecule in the L\"owdin basis for the chosen level of theory. This matrix, together with the list of atoms each L\"owdin orbital is attached to, is exported in a file.
    \item[Step 2:] The rest of the calculation consists in:
    \begin{itemize}
\item the computation from the chosen fragmentation and the 1-RDM $\gamma_{\rm GS}$ of the $A$, $B$, $C$ matrices of the bath construction problems associated with each fragment, using formulae \eqref{rdm1_basis1}--\eqref{eq:BCP-C},
\item the computation of the numerical solutions of~\eqref{eq:OptProblem}--\eqref{eq:CostFunctionJ} and \eqref{eq:convexification} by the different algorithms presented in Section~\ref{sec:numerical_methods}.
\end{itemize}
\end{description}

For the sake of brevity, we only report here calculations on the benzene molecule (C$_6$H$_6$), in which the 12 atoms are partitioned into 6 identical fragments, each fragment being composed of a carbon atom and the hydrogen atom covalently bonded to it (see Fig.~\ref{fig:physical_systems} (a)). We used the STO-3G minimal orbital basis set, for which there is a total of $L=36$ orbitals. Each fragment contains $\ell=6$ orbitals, so that we have for this example $M=30$. We computed the ground-state 1-RDM $\gamma_{\rm GS}$ at the CISD level of theory. As we imposed the symmetry in the PySCF calculation, $\gamma_{\rm GS}$ has the symmetry of the benzene molecule. In particular, all six bath construction problems are identical. As expected, the trace of $\gamma_{\rm GS}$ is equal to 21, the number of electron pairs in the benzene molecule. We have $\|\gamma_{\rm GS}^2-\gamma_{\rm GS}\| \simeq 0.067$,
which shows that $\gamma_{\rm GS}$ is close to an orthogonal projector, which, in terms of chemical properties, corresponds to the 1-rdm of a ground-state closed to a Slater determinant, and testifies of a weakly-correlated system.

\medskip

Our first step here is to show the consistency of the numerical results with the theoretical ones, For instance, we want to check that the assumption of Theorem \ref{thm:convex_pbm} do hold in this concrete case. According to Theorem \ref{thm:convex_pbm}, if the gradient of the convex problem \eqref{eq:convexification} $\widetilde{G}_\star$ shows a spectral gap between its $m$-th and $(m+1)$-st eigenvalues, both problems \eqref{eq:OptProblem} and \eqref{eq:convex_opt} have a unique global minimizer.
Hence, to check the assumptions of this theorem, the gap $\mu_{m+1}-\mu_m$ and the Frobenius norm $\|D_\star^2-D_\star\|$, where $\mu_1 \le \cdots \le \mu_M$ are the eigenvalues of $\widetilde G_\star:=\nabla \widetilde J(D_\star)$ and $D_\star$ is the found minimizer of the convexified problem  \eqref{eq:convexification}, are plotted as a function of the bath dimension $m$ in Fig~\ref{fig:benzene_convexified}. We see that, by Theorem~\ref{thm:convex_pbm}, ODA provides the global minimizer $P_\star$ of~\eqref{eq:OptProblem}--\eqref{eq:CostFunctionJ} for $m \le 5$.

 \begin{figure}[ht]
    \centering
    \begin{minipage}{0.49\textwidth}
        \includegraphics[width=1.0\textwidth]{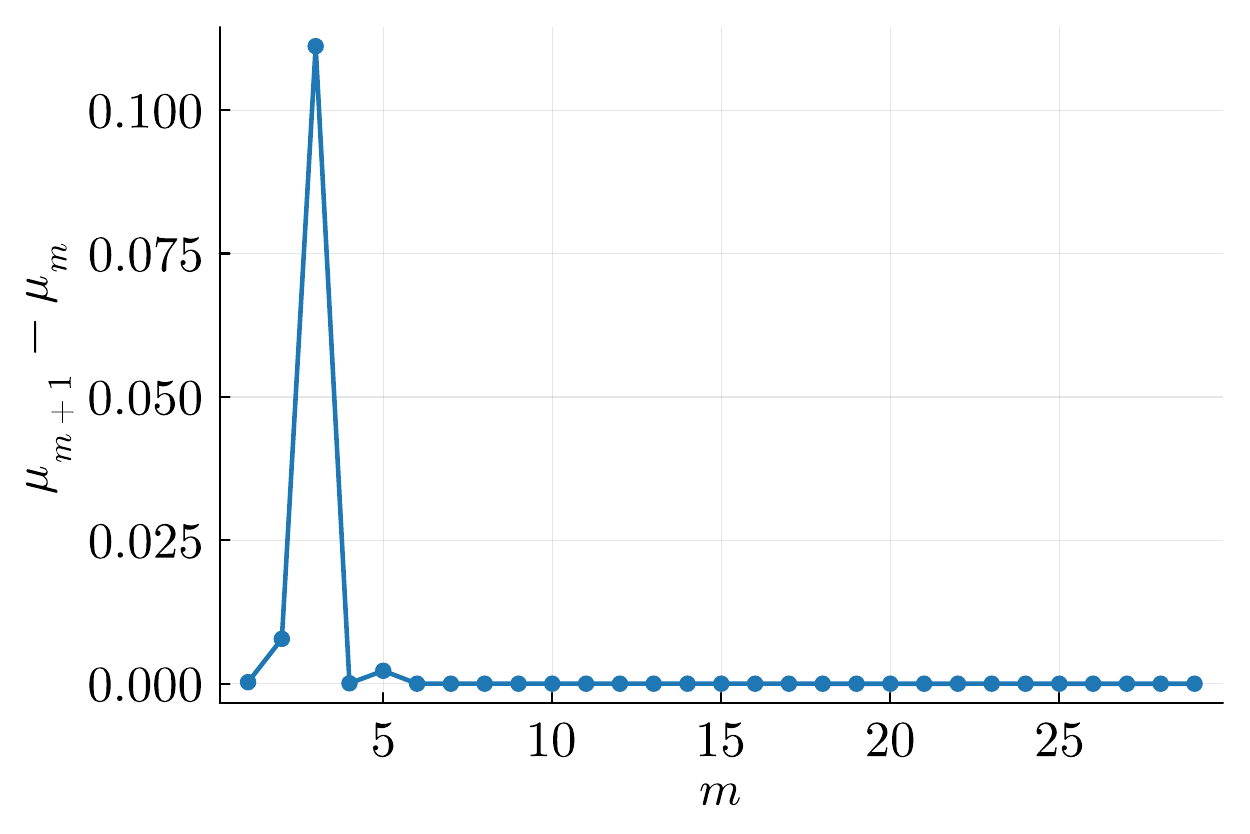}
    \end{minipage}
    \hfill
    \begin{minipage}{0.49\textwidth}
        \includegraphics[width=1.0\textwidth]{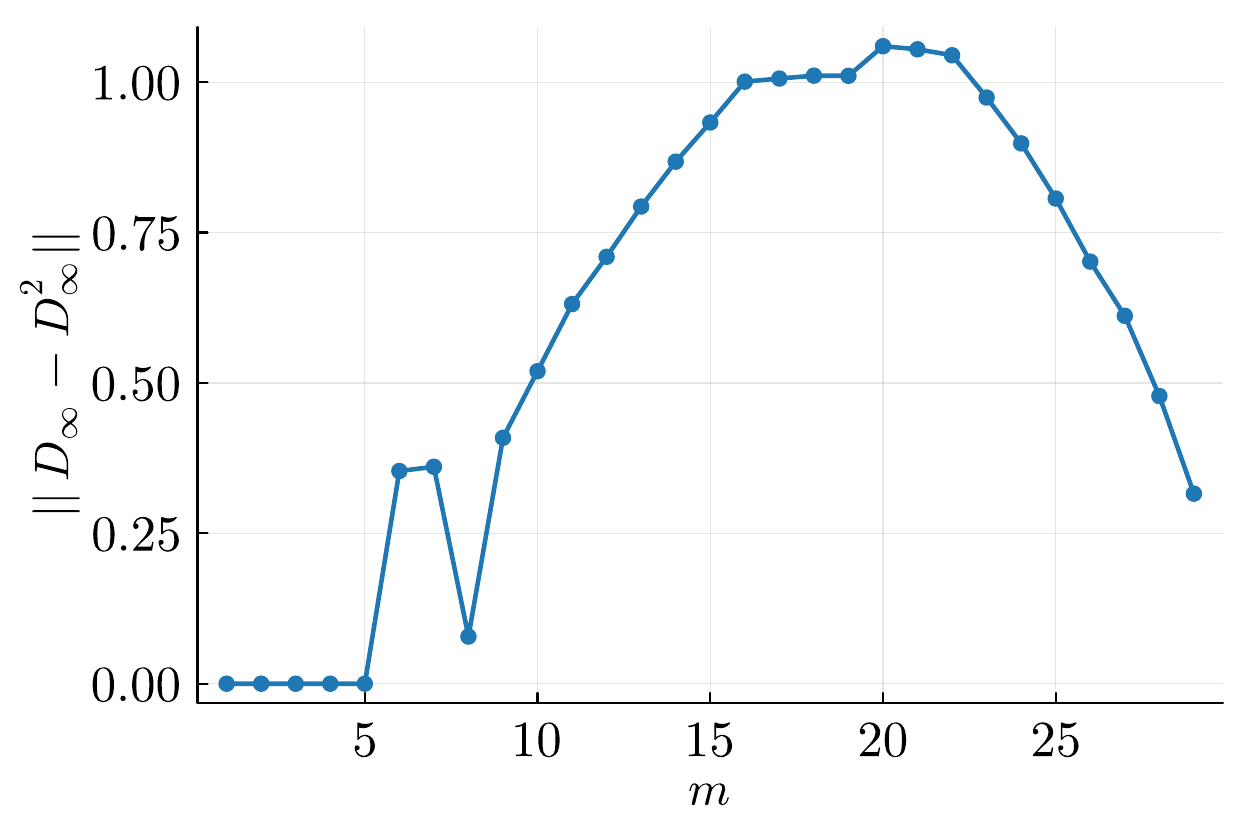}
    \end{minipage}
    \caption{Gap $\mu_{m+1}-\mu_m$ and  non-idempotency criterion $\|D_\star - D_\star^2\|$ as functions of the bath dimension $m$ for the benzene molecule. \label{fig:benzene_convexified}}
\end{figure}

Convergence rates of the SCF and Grassmann trust-region algorithms for $m=3$, $m=15$ and $m=18$ are displayed in Figs.~\ref{fig:benzene_m=3}, \ref{fig:benzene_m=15_cost_radius} and \ref{fig:benzene_m=18_cost_radius}, respectively. For $m=3$, we also show the ODA convergence rate since this algorithm provides the global minimum for $m \le 5$ ($ \mathcal{J}_{min}^\gamma  \simeq  0.0094538$ for $m=3$). When initialized with the natural initial guess $P_0=\1_{(-\infty,c_m]}(C)$ (the global minimizer when $A=0$), all algorithms quickly converge to the global minimum. For $m=18$, the SCF algorithm initialized with $P_0=\1_{(-\infty,c_m]}(C)$ has a long plateau about $1 \times 10^{-3}$ higher than the assumed global minimum ($\mathcal{J}_{min}^\gamma  \simeq 7.3399756 \times 10^{-7}$ for $m=18$), while  the trust-region algorithms used on the Grassmann manifold also display long plateaus about $5 \times 10^{-5}$ higher than the assumed global minimum. Taking as initial guess one of the two accumulation points of the sequence $(P_k^{\rm ODA})_{k \in \N}$ generated by the ODA algorithm (denoted abusively by $P_\infty^{\rm ODA}$), the SCF algorithms display a long plateau, but only slightly above ($\sim 10^{-6}$) the assumed global minimum, while the Riemannian trust-region algorithm quickly reaches the assumed global minimum.

\begin{figure}[ht]
    \centering
    \includegraphics[width=0.6\textwidth]{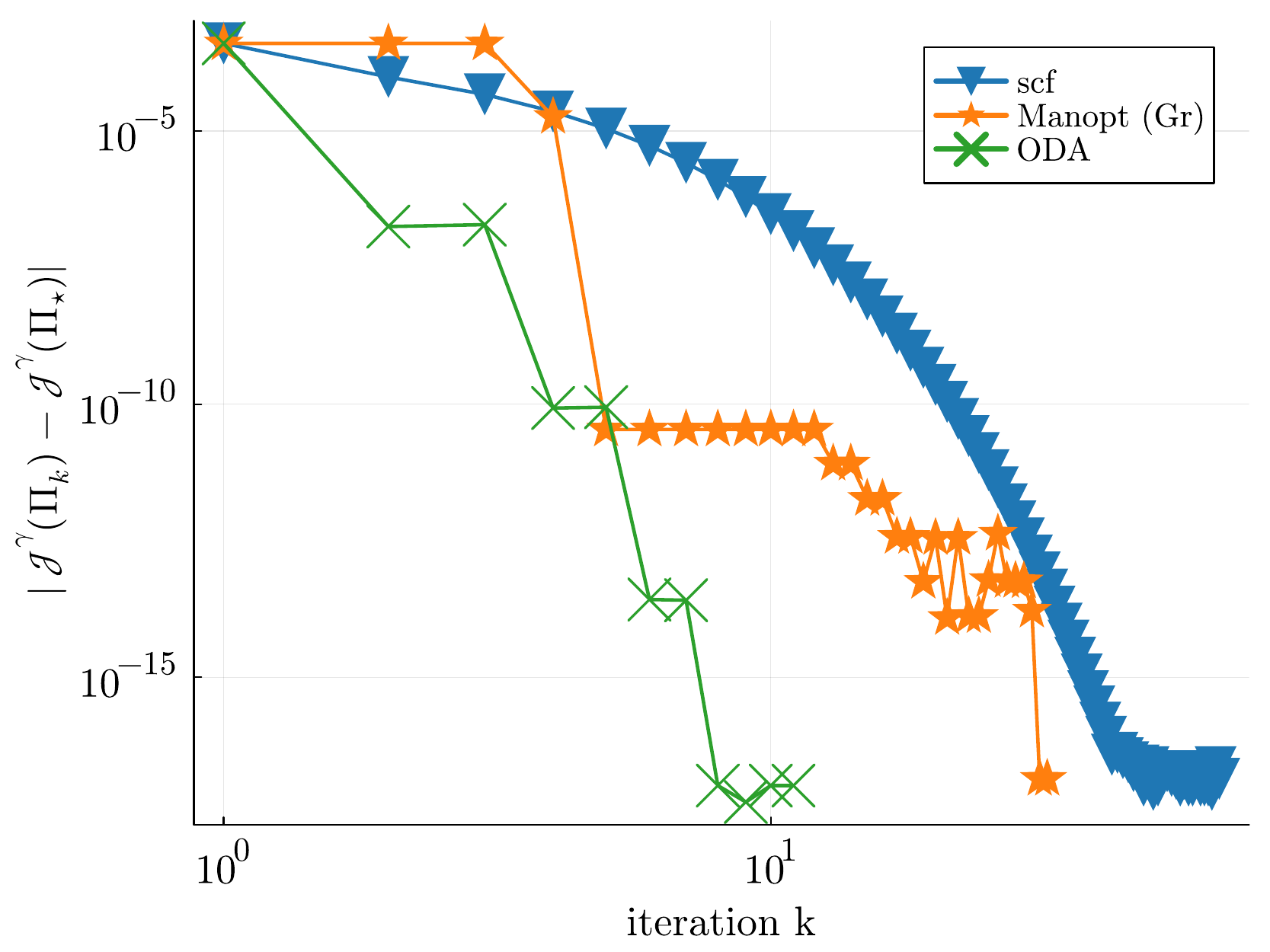}
    \caption{Decay of the cost function  $\mathcal{J}^\gamma$ for benzene, a bath dimension $m=3$, various numerical algorithms, and the inital guess $P_0 = \1_{(-\infty,c_{m}]}(C)$. For $m=3$, the assumed global minimum is $ \mathcal{J}_{min}^\gamma  \simeq  0.0094538$ \label{fig:benzene_m=3} }
\end{figure}

 \begin{figure}[ht]
    \centering
    \begin{minipage}{0.49\textwidth}
            \includegraphics[width=1.0\textwidth]{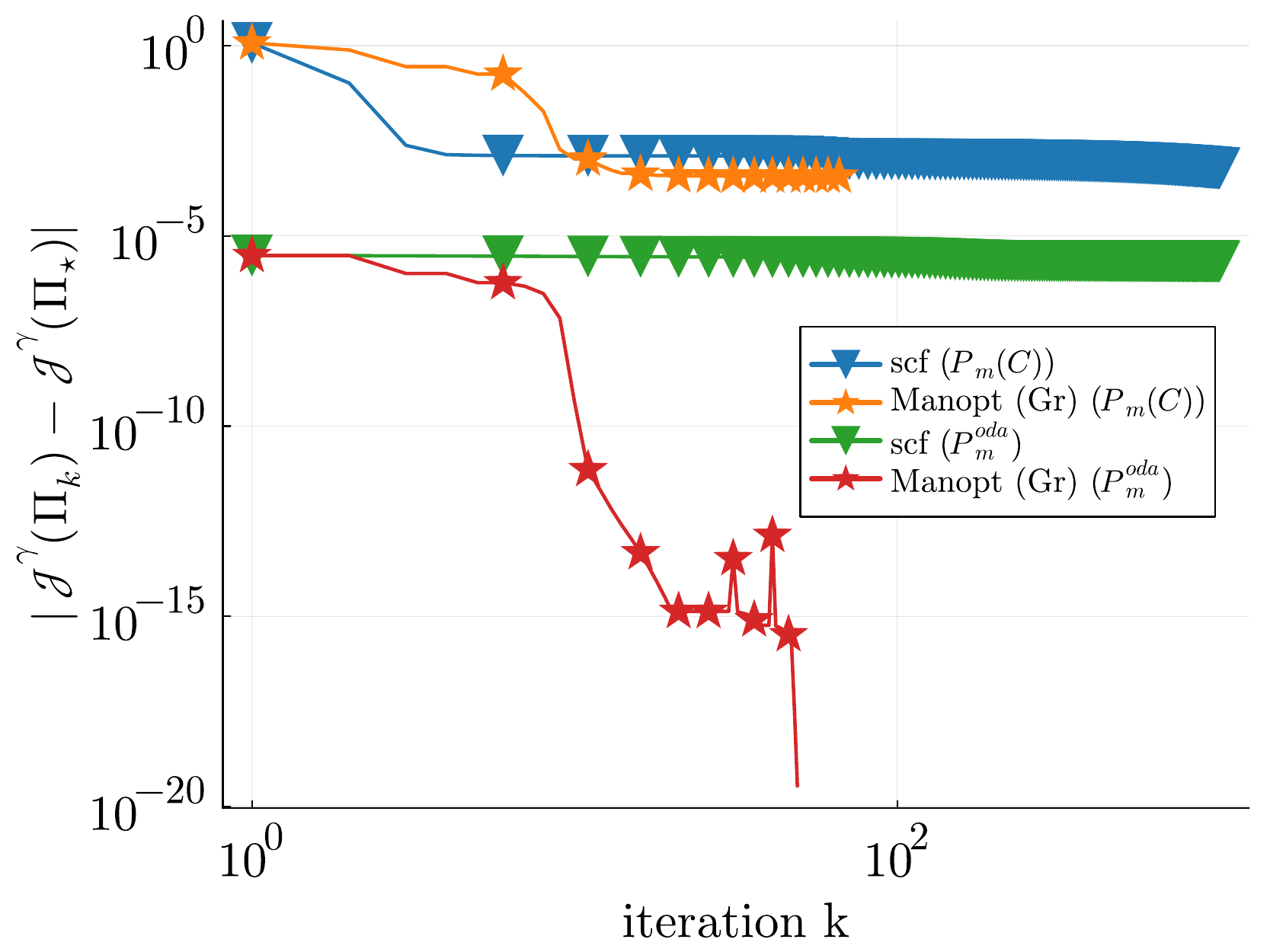}
    \end{minipage}
    \hfill
    \begin{minipage}{0.49\textwidth}
           \includegraphics[width=1.0\textwidth]{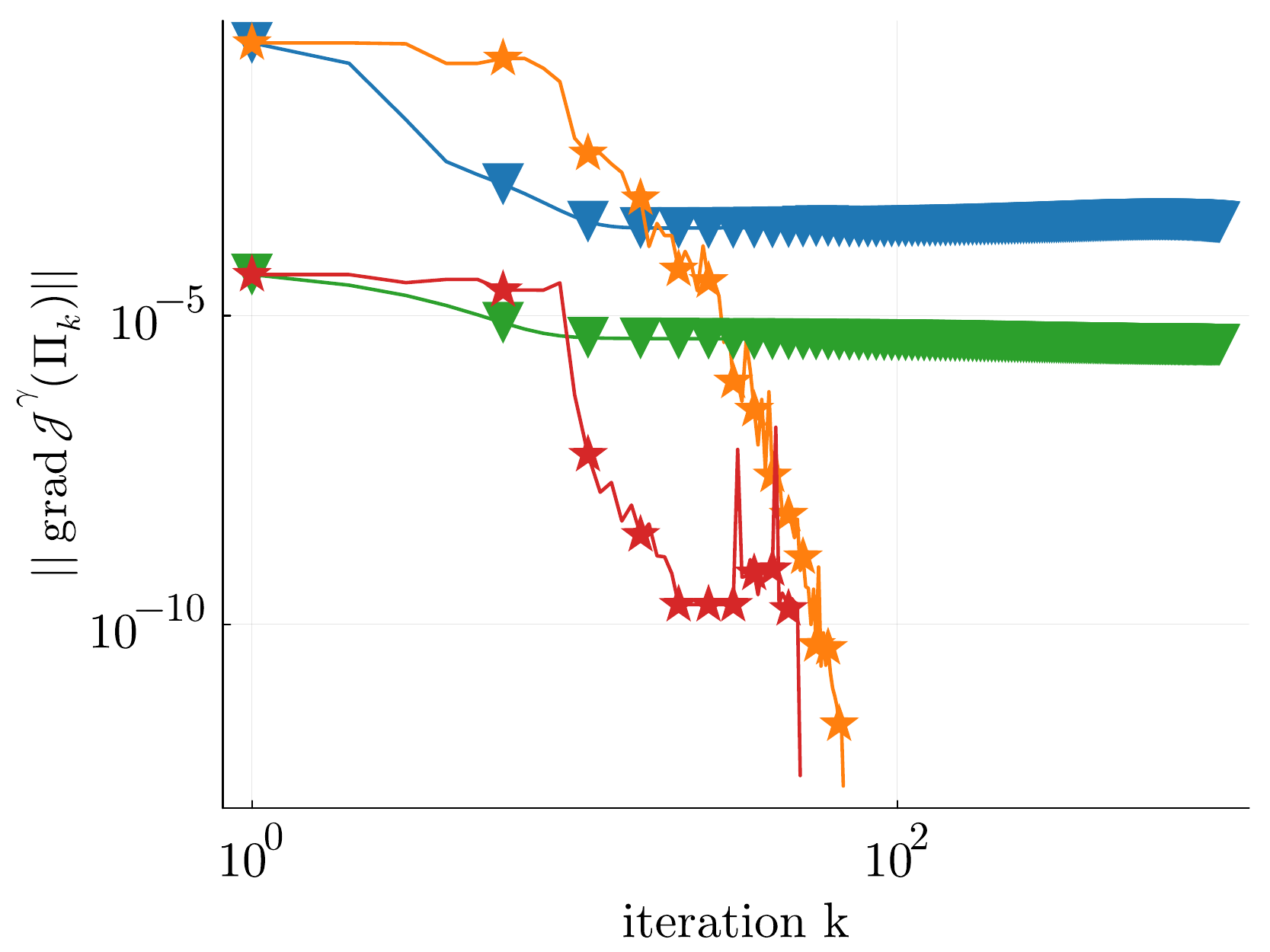}
    \end{minipage}
    \caption{Decay of the cost function  $\mathcal{J}^\gamma$ (left) and decay of the riemannian gradient norm  $\text{grad}_{\mathcal M}\mathcal{J}^\gamma$ (right) for benzene, a bath dimension $m=15$, various numerical algorithms, and two different initializations $P_0=\1_{(-\infty,c_m]}(C)$ and $P_0=P_\infty^{\rm ODA}$. 
    \label{fig:benzene_m=15_cost_radius} }
\end{figure}

 \begin{figure}[ht]
    \centering
    \begin{minipage}{0.49\textwidth}
            \includegraphics[width=1.0\textwidth]{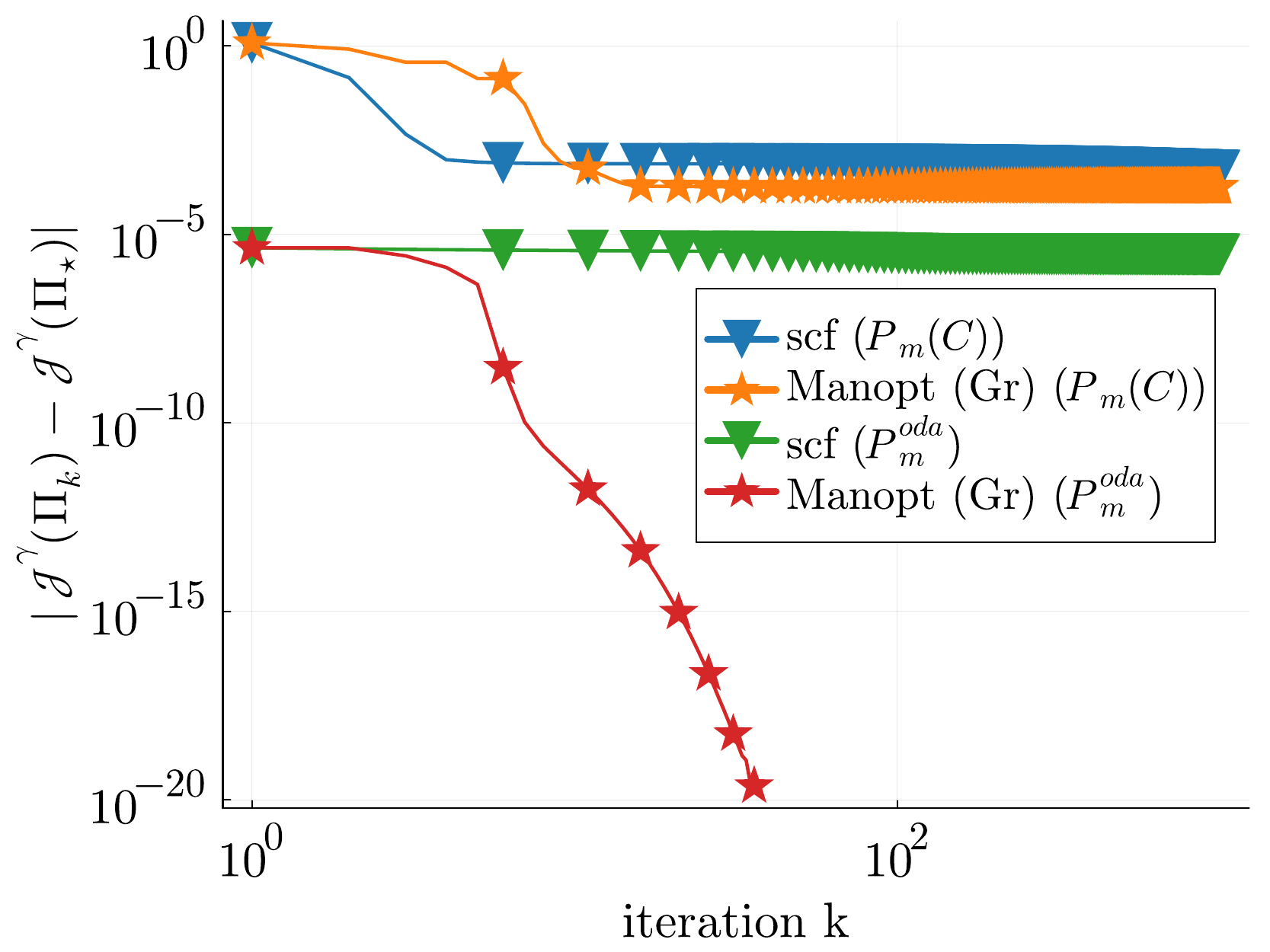}
    \end{minipage}
    \hfill
    \begin{minipage}{0.49\textwidth}
        \includegraphics[width=1.0\textwidth]{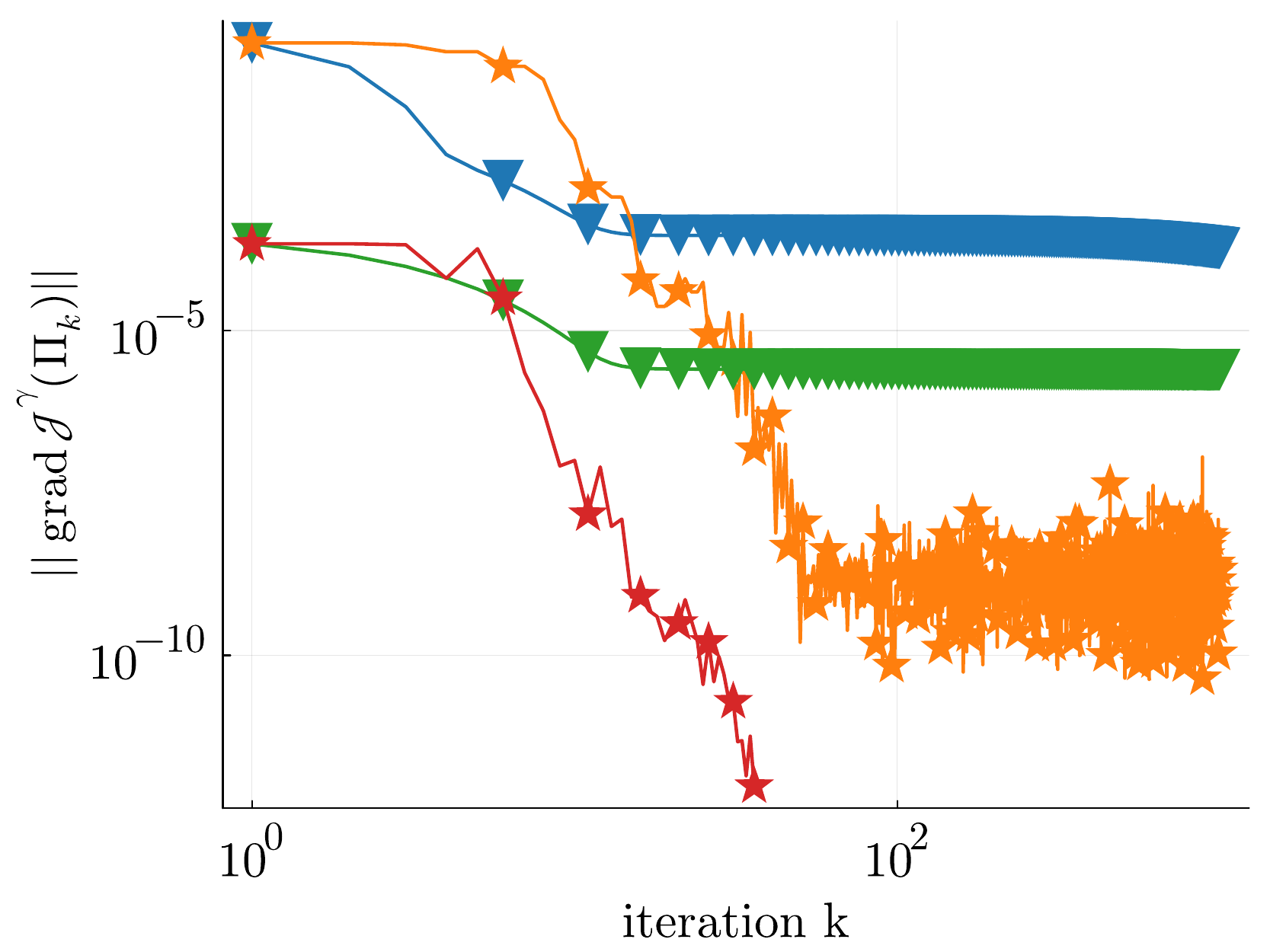}
    \end{minipage}
    \caption{Decay of the cost function  $\mathcal{J}^\gamma$ (left) and decay of the riemannian gradient norm  $\text{grad}_{\mathcal M}\mathcal{J}^\gamma$ (right) for benzene, a bath dimension $m=18$, various numerical algorithms, and two different initializations $P_0=\1_{(-\infty,c_m]}(C)$ and $P_0=P_\infty^{\rm ODA}$. 
    \label{fig:benzene_m=18_cost_radius} }
\end{figure}

\subsubsection{1D Hubbard chain, 4 sites}
\label{sec:Hubbard}

The Hubbard model is one of the most important models in condensed-matter physics. It is a quantum lattice model used to approximate, study and understand strongly-correlated materials and their behaviors, such as high-temperature superconductivity or the so-called Mott transition between a metal and a Mott insulator.

\noindent
The Hamiltonian of the single-orbital Hubbard model reads, in a second-quantized form:
\begin{equation*}
    \hat{H}_{\rm H} = - t \sum_{\langle i,j  \rangle \in \mathcal N} \sum_{\sigma \in \{\uparrow,\downarrow\}}(\hat{a}_{i,\sigma}^{\dagger} \hat{a}_{j,\sigma} + \hat{a}_{j,\sigma}^{\dagger} \hat{a}_{i,\sigma} ) + U \sum_{i \in \mathcal S} \hat{n}_{i,\uparrow} \hat{n}_{i,\downarrow} - \mu \sum_{i \in \mathcal S} ( \hat{n}_{i,\uparrow} + \hat{n}_{i,\downarrow}),
\end{equation*}
where $\mathcal S$ is the set of lattice sites (a site represents a pair of spin-orbitals sharing the same spatial orbital), $\mathcal N$ the set of undirected edges linking two neighbors, $t \in \R_+$ the hopping parameter, $\mu$ the chemical potential of the electrons, $\hat{a}_{i,\sigma}$ and $\hat{a}_{i,\sigma}^{\dagger}$ the annihilation and creation operators at site~$i$ and spin $\sigma$, with canonical anticommutation relations (CAR)
\begin{equation*}
\begin{split}
        \forall i,j \in \mathcal S, \quad\forall \sigma, \sigma' \in \{\uparrow,\downarrow\}, \quad
    \{\ \hat{a}_{i,\sigma}^{\dagger}, \hat{a}_{j,\sigma'}^{\dagger}\} &= \{\hat{a}_{i,\sigma}, \hat{a}_{j,\sigma'}\} = 0, \quad
    \{\hat{a}_{i,\sigma},\hat{a}_{j,\sigma'}^{\dagger}\}  = \delta_{i,j}\delta_{\sigma,\sigma'}.
\end{split}
\end{equation*}
Finally $\hat{n}_{i,\sigma}$ is the number operator of the quantum state $(i,\sigma)$ defined by
\begin{equation*}
    \hat{n}_{i,\sigma} = \hat{a}_{i,\sigma}^{\dagger}\hat{a}_{i,\sigma}.
\end{equation*}
In this model, a site $i$ can either be unoccupied, or occupied by a spin-up or a spin-down electron, or by a pair of electrons of opposite spins. Electrons can jump from one site to any other neighboring site (tunnel effect), with a gain in energy $-t$. Besides, electrons occupying the same site repel each other with a repulsion energy $U \in \R_+$. 

The Hubbard model is notoriously difficult to solve because the hopping term and the interaction term---favoring delocalized and localized electrons, respectively---cannot be diagonalized simultaneously and efficiently, and an exact numerical diagonalization of the full matrix is possible only for small sizes $N_s$ due to the exponential growth of the Hilbert space.

In the following, we will be working at half-filling, namely with an average number of electrons per site of one (in practice, this is achieved by setting $\mu =U/2$). It is the filling at which the effect of correlations is supposed to be the strongest. In particular, when the interaction-to-hopping ratio $U/t$ is increased, one expects the ground state to turn from a delocalized wavefunction (indicative of a metallic behavior) to a localized one (indicative of a so-called Mott insulator). In the thermodynamic limit, one expects a phase transition to happen at a critical $U/t$ ratio, while at the finite sizes we will consider here, it is only a crossover.

\medskip

Here, we tackle the 1D Hubbard with DMET to explore the behavior of our Riemannian optimization problem. In particular, in our new DMET construction, the bath size $m$, contrary to standard DMET, does not equal the fragment size. We are interested in the influence of this bath size on the behavior of the Riemannian optimization problem, and in particular whether the cost function $\mathcal{J}^\gamma$ defined in~\eqref{eq:J_gamma} can converge to 0 for some bath size $m$, which would signal a perfect disentangling of the impurity problem from its environment.

For the Hubbard model, a straightforward DMET fragmentation consists in considering that each site is a fragment. Each fragment thus contains two spin-orbitals sharing the same spatial orbital. This choice corresponds to $\ell=2$ and $M=2(N_{\rm s}-1)$ in our notation.

In the half-filled sector, we restrict ourselves to paramagnetic eigenstates, which are eigenvectors of the $\hat{S}_z$ observable associated the 0 eigenvalue.

\begin{figure}[ht]
    \centering
    \begin{tikzpicture}[scale=0.5,>=stealth]

\def\n{4}
\def\d{3}

\definecolor{sitecol}{RGB}{100,180,100}
\definecolor{tcol}{RGB}{200,60,100}
\definecolor{Ucol}{RGB}{50,90,200}
\definecolor{sitepale}{RGB}{190,225,190}
\definecolor{tpale}{RGB}{235,180,195}
\definecolor{Upale}{RGB}{175,195,240}

\draw[rounded corners=18pt, draw=gray!50, line width=1pt, fill=gray!8]
  ({1*\d-1.2},-1.5) rectangle ({(\n-1)*\d+1.2},1.5);

\foreach \i in {1,...,\n}
{
  \ifnum\i=1
    \node[circle, draw, fill=sitecol, minimum size=1.2pt] (s\i) at ({(\i-1)*\d},0) {};
    \draw[Ucol, line width=1.1pt] (s\i) circle (0.48);
    \draw[black, line width=1.4pt]
      ({(\i-1)*\d-0.85},-0.85) rectangle ({(\i-1)*\d+0.85},0.85);
    \node[Ucol] at ({(\i-1)*\d},1.4) {\Large $U$};
  \else
    \node[circle, draw=gray!50, fill=sitepale, minimum size=1.2pt] (s\i) at ({(\i-1)*\d},0) {};
    \draw[Upale, line width=1.1pt] (s\i) circle (0.48);
    \node[Upale] at ({(\i-1)*\d},1.4) {\Large $U$};
  \fi
}

\foreach \i in {1,...,\numexpr\n-1\relax}
{
  \pgfmathtruncatemacro{\next}{\i+1}
  \ifnum\i=1
    \draw[<->, very thick, tcol] (s\i) -- (s\next);
    \path (s\i) -- (s\next) coordinate[midway] (m);
    \node[tcol, above] at (m) {\Large $t$};
  \else
    \draw[<->, very thick, tpale] (s\i) -- (s\next);
    \path (s\i) -- (s\next) coordinate[midway] (m);
    \node[tpale, above] at (m) {\Large $t$};
  \fi
}

\end{tikzpicture}

    \caption{The approach for the numerical example considers the leftmost fragment.}
    \label{fig:hubbard_frag}
\end{figure}
In the simulations reported below, we consider the leftmost site of a four-site Hubbard chain (Fig.~\ref{fig:physical_systems} (b)), as our fragment. We proceed as follows:
\begin{description}
    \item[\textbf{Step 1}:] Considering the symmetries of the problem, we use a sparse diagonalization method (from scipy.sparse) for the real-symmetric Hamiltonian, considering the half-filled sector, with number of electrons $N_s$, of size ${L \choose N_{\rm s} }$, and also choose the $S_Z = 0$ sector, to obtain the ground state, and construct its exact ground-state 1-RDM $\gamma_{} \in \Gamma_{\mathrm{MS}}$. The cost of this step is exponential in $L$ and thus limits us to small lattices.
    \item[\textbf{Step 2}:] We solve the DMET bath construction problem for the natural fragmentation described above and the exact ground-state 1-RDM $\gamma$ computed at the previous step.
  \end{description}

\begin{figure}
    \centering
    \includegraphics[width=0.6\linewidth]{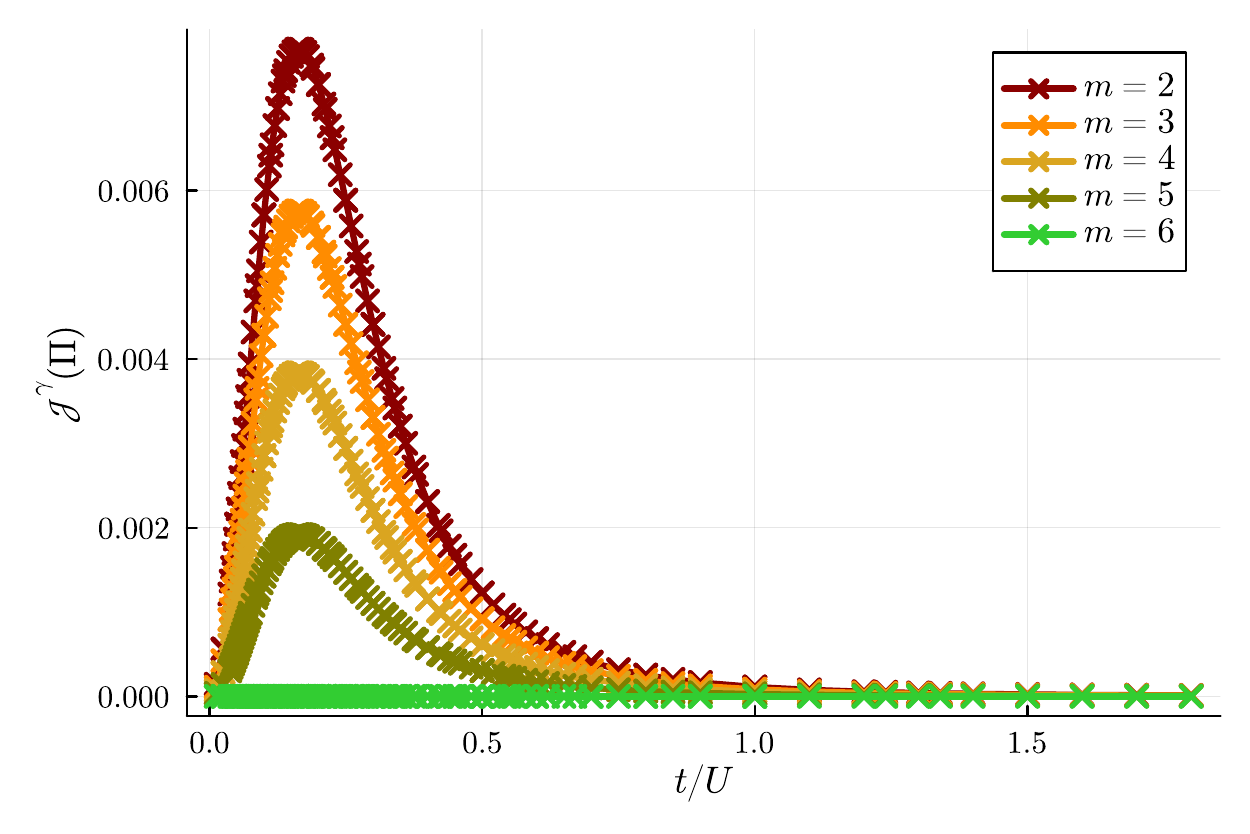}
    \caption{Evolution of the cost function for the fragment composed of the leftmost site in the 4-site 
     1-D chain Hubbard model as a function of the ratio $t/U$, for several bath dimensions $m$. 
    }
    \label{fig:hubbard-4s}
\end{figure}

In Fig.~\ref{fig:hubbard-4s}, we show the evolution of the cost function as a function of $t/U$ for the fragment composed of the leftmost site. We see that, for all values of the bath size $m$, $\mathcal{J}^\gamma$ gradually increases from 0 at $t/U=0$ up to a maximum at $t/U \approx 0.17$, and then decreases back to 0 at large $t/U$ ratios. 
Furthermore, we observe that the larger $m$, the smaller the cost function. When the bath size is such that the bath encompasses the whole environment of the fragment ($m=M$ in our notations),  $\mathcal{J}^\gamma$ vanishes. 

We can interpret this evolution as follows:
\begin{description}
    \item[Metallic regime:] in the $t/U \rightarrow \infty$ limit, the system is mainly driven by the kinetic energy: the 1-RDM becomes a Slater determinant and the full disentanglement for the bath construction reduces to a QR or SVD decomposition, as emphasized in Proposition~\ref{prop:DMET_proj_sol};
    \item[Insulating regime:] in the $t/U \rightarrow 0$ limit, fluctuations are weak and therefore only a small bath is needed to describe them.
\end{description}

In between, we cannot completely disentangle the fragment+bath system from the environment, but this disentanglement becomes better and better as the bath size increases, as expected. The point at which this disentanglement is hardest ($t/U \approx 0.17$) likely corresponds to the metal-to-Mott crossover at this lattice size.

We note that the behavior of $\mathcal{J}^\gamma$ depends not only on the input 1-RDM $\gamma$ but also on the chosen fragment as it determines he block decomposition of $\gamma$ and so the matrices $A$, $B$ and $C$. We also note that with periodic boundary conditions the $t/U\rightarrow \infty$ limit becomes less trivial as even the noninteracting ground state in the half-filled paramagnetic sector becomes degenerate and therefore possibly not a single Slater determinant. This in turn leads to a non-idempotent 1-RDM $\gamma$ and a nonvanishing $\mathcal{J}^\gamma$ for small bath sizes $m$.

 \begin{figure}
    \centering
    \begin{minipage}{0.48\textwidth}
            \includegraphics[width=1.0\textwidth]{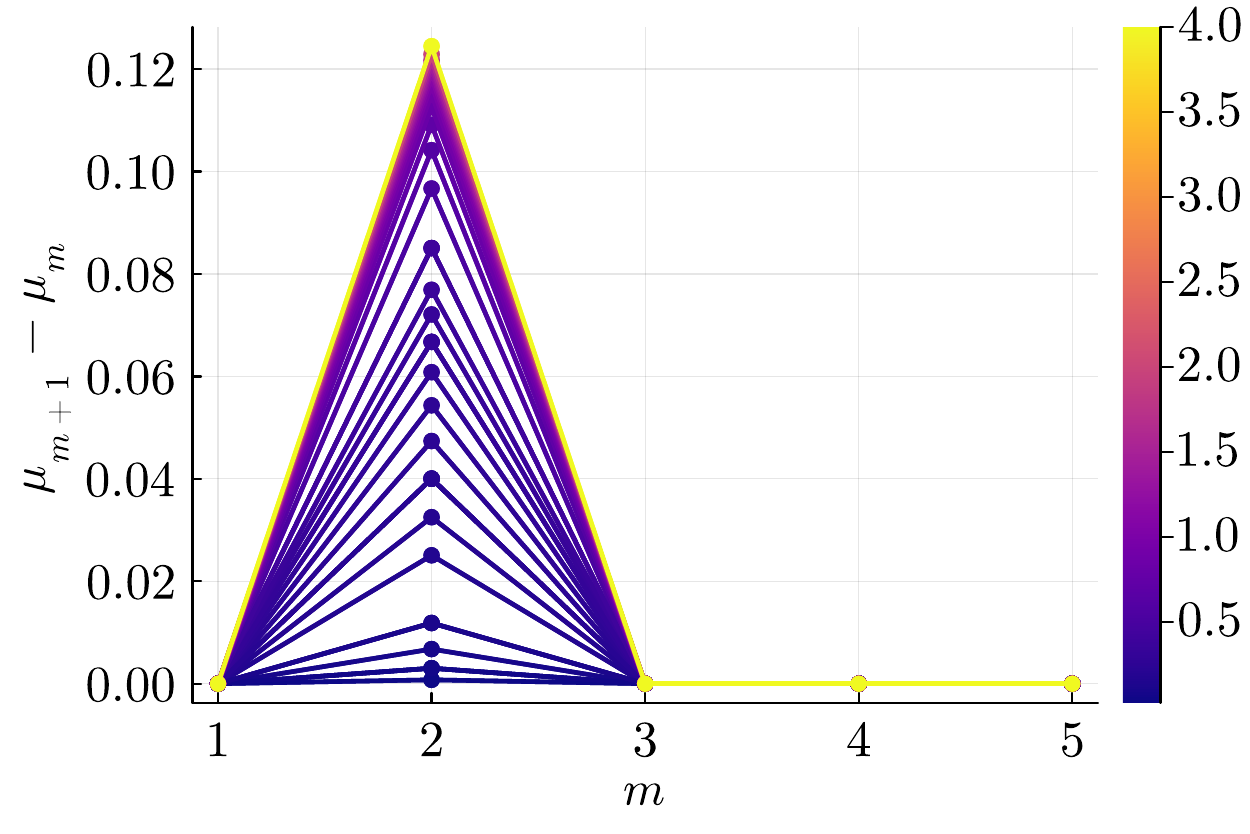}
    \end{minipage}
    \hfill
    \begin{minipage}{0.49\textwidth}
            \includegraphics[width=1.0\textwidth]{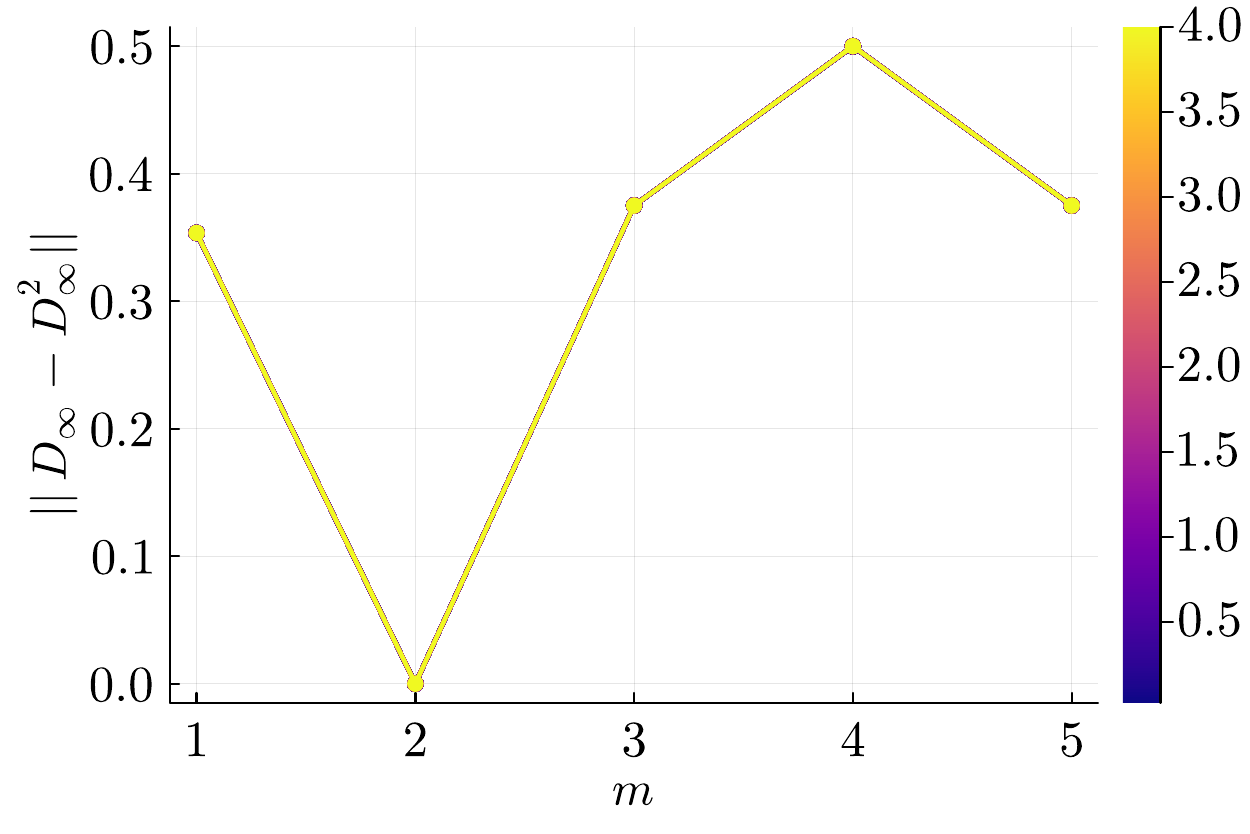}
    \end{minipage}
    \caption{Gap $\mu_{m+1}-\mu_m$ (left) and  non-idempotency criterion $\|D_\infty - D_\infty^2\|$ (right) as functions of the bath dimension $m$ for the fragment composed of the leftmost site the 4-site Hubbard model, for different values of the ratio $t/U$. }\label{fig:Hubbard_convexified}
\end{figure}

\medskip

Again, we rely on the theoretical aspects presented in Section~\ref{sec:convexification} to better understand the behavior of the various algorithms. 
We recall that, according to Theorem \ref{thm:convex_pbm}, if the gradient of the convex problem \eqref{eq:convexification} $\widetilde{G}_\star$ shows a spectral gap between its $m$-th and $(m+1)$-st eigenvalues, then Problem \eqref{eq:OptProblem} has a unique minimizer.
Hence, by examining the spectral gap, we can determine the value of the bath size $m$ that ensures uniqueness. In Figure~\ref{fig:Hubbard_convexified}, we thus plot this spectral gap as well as the distance of the solution $D_\infty$ given by the Frank--Wolfe/ODA scheme $\eqref{eq:ODA}$ for $\widetilde J$ to the target Grassmann manifold $\mathcal{M}$ as a function of the ratio $t/U$.
We observe that, except for the insulating regime where all the electrons are localized, there is only one bath dimension for which we can be certain to have a unique solution, namely $m=2$. It is notable that the bound~\eqref{eq:coercivity} on the coercivity constant of the function $J$ at the minimizer $P_\star$, which is proportional to the gap $\mu_{m+1}-\mu_m$, goes to zero in the insulating regime. Indeed, in the insulating phase where all the electrons are localized, there is no unique way to create a bath of dimension $m=2$.

\section{Conclusion}

The quadratic Grassmann optimization problem~\eqref{eq:OptProblem}--\eqref{eq:CostFunctionJ} arising from the DMET bath-construction problem~\eqref{eq:opt_Pi} has interesting mathematical and numerical properties. As the Hartree--Fock problem - a famous quadratic Grassmann optimization problem -, it satisfies an Aufbau principle (Theorem~\ref{thm:Aufbau}). A major qualitative difference between these two problems is that the basic Roothaan self-consistent field (SCF) algorithm always converges (if $A$ is invertible) to a critical point of~\eqref{eq:OptProblem}--\eqref{eq:CostFunctionJ} satisfying the Aufbau principle (Proposition~\ref{prop:SCF}), whereas for Hartree--Fock, it often oscillates between two points which are not critical points of the Hartree--Fock functional~\cite{cances2000b}. Another interesting specificity of the cost function $J$ defined in  \eqref{eq:CostFunctionJ} is that the {\em convex} quadratic function $\widetilde J$ defined in~\eqref{eq:def_tildeJ} takes the same values as $J$ on the Grassmann manifold $\mathcal M$. Minimizing $\widetilde J$ on the convex hull ${\rm CH}(\mathcal M)$, using e.g. the optimal damping algorithm (ODA), yields a global minimizer $D_\star \in {\rm CH}(\mathcal M)$, and the matrix $\widetilde G_\star := \nabla \widetilde J(D_\star)$ is independent of the found minimizer. The results in Theorem~\ref{thm:convex_pbm} show that if $\widetilde G_\star$ satisfies the gap condition $\mu_m < \mu_{m+1}$, then $D_\star$ is a {\em global} minimizer of the non-convex problem~\eqref{eq:OptProblem}--\eqref{eq:CostFunctionJ}. Otherwise, the numerical tests we have run so far seem to indicate that $D_\star$ provides a good starting point for Riemannian optimization methods to reach the assumed global minimum of $J$ on~$\mathcal M$.

\section*{Acknowledgements} This project has received funding from the Simons Targeted Grant in Mathematics and Physical Sciences on Moir\'e Materials Magic (Award No. 896630, E.C., L.L.), the European Research Council (ERC) under the European Union's Horizon 2020 research and innovation program (grant agreement EMC2 No 810367) and the EPIQ PEPR project (ANR-22-PETQ-0007). The authors thank Nicolas Boumal, Laura Grazioli, Gaspard Kemlin and Raehyun Kim for useful discussions. 

\printbibliography

\end{document}